\documentclass[a4paper,reqno]{amsart}
\usepackage{amssymb}
\usepackage{amsmath,amssymb,mathrsfs,latexsym,tikz,cite,hyperref}
\usepackage{genyoungtabtikz}
\usetikzlibrary{decorations.pathreplacing}
\usepackage{rotating, graphicx}
\usepackage{lscape,marginnote}

\title{Rouquier blocks for Ariki-Koike algebras}
\author[S.~Lyle]{Sin\'ead Lyle}
\address{School of Mathematics, University of East Anglia, Norwich NR4 7TJ, UK.}
\email{s.lyle@uea.ac.uk}
\thanks{This research was undertaken during the period that the author was an LMS Emmy Noether Fellow in Mathematics. We thank the LMS for their support.}

\numberwithin{equation}{section}
\numberwithin{figure}{section}
\newtheorem{lemma}{Lemma}[section]
\newtheorem{theorem}[lemma]{Theorem}
\newtheorem{proposition}[lemma]{Proposition}
\newtheorem{corollary}[lemma]{Corollary}

\theoremstyle{remark}

\newtheorem*{ex}{Example}

\newcommand{\Z}{\mathbb{Z}}
\newcommand{\la}{\lambda}
\newcommand{\La}{\Lambda}
\newcommand{\emp}{\varnothing}
\newcommand{\Laz}[1]{\Lambda^{(#1)}}
\newcommand{\Larn}{\Lambda^{(r)}_n}
\newcommand{\Lar}{\Lambda^{(r)}}
\newcommand{\Kl}{\overline{\La}^{(r)}_n}
\newcommand{\Kla}{\overline{\La}^{(r)}}

\newcommand{\R}{\mathscr{R}}
\newcommand{\Ra}{\mathcal{R}}

\newcommand{\A}{\mathscr{A}}
\newcommand{\h}{\mathcal{H}_{r,n}}
\newcommand{\hn}{\mathcal{H}_{n}}
\newcommand{\hk}{\mathrm{h}}
\newcommand{\mc}{\boldsymbol a}
\newcommand{\sd}{\sim_{e,\mc}}
\newcommand{\ap}{\approx}

\newcommand{\bs}{\boldsymbol{s}}
\newcommand{\bt}{\boldsymbol{t}}
\newcommand{\bla}{\boldsymbol{\la}}
\newcommand{\bmu}{\boldsymbol{\mu}}
\newcommand{\brho}{\boldsymbol{\rho}}

\newcommand{\M}{{\boldsymbol M}}
\DeclareMathOperator{\rad}{rad}
\DeclareMathOperator{\res}{res}
\DeclareMathOperator{\Res}{Res}
\DeclareMathOperator{\im}{Im}
\DeclareMathOperator{\wt}{wt}
\DeclareMathOperator{\Str}{Str}
\definecolor{bead}{gray}{0.2}

\begin{document}
\begin{abstract}
The Rouquier blocks, also known as the RoCK blocks, are important blocks of the symmetric groups algebras and the Hecke algebras of type $A$, with the partitions labelling the Specht modules that belong to these blocks having a particular abacus configuration. We generalise the definition of Rouquier blocks to the Ariki-Koike algebras, where the Specht modules are indexed by multipartitions, and explore the properties of these blocks. 
\end{abstract}
\maketitle

\section{Introduction}
Suppose there is a conjecture that you believe to be true for all finite groups. If no general method of attack suggests itself, perhaps you begin by proving it for specific families of groups. You might start with the abelian groups, or try to work your way through the simple groups, but fairly soon it is likely you will want to consider the symmetric groups $\mathfrak{S}_n$. One may easily see examples of this approach. In representation theory, conjectures for which partial proofs exist include Alperin's weight conjecture (conjectured in 1987~\cite{Alperin} and proved for symmetric groups in 1990~\cite{AlperinFong}), Donovan's conjecture (conjectured in 1980~\cite[Conjecture M]{AlperinBook} and proved for symmetric groups in 1991~\cite{Scopes}) and Brou\'e's abelian defect group conjecture (conjectured in 1988 and proved for symmetric groups in 2008~\cite{CK,CR}). All of these conjectures are still open in general. 

Now suppose you want to prove a theorem about the representations of the symmetric groups. Again, if no general approach is obvious, you may think about proving the result for specific blocks. A classic example of a proof following this method is the proof of Brou\'e's abelian defect group conjecture for symmetric groups. Chuang and Kessar~\cite{CK} took certain blocks defined by Rouquier~\cite{Rouquier} and proved that Brou\'e's conjecture holds for these blocks by showing that any such block of weight $w<p$ is Morita equivalent to the principal block of $\mathfrak{S}_p \wr \mathfrak{S}_w$. Subsequently Chuang and Rouquier extended the proof of Brou\'e's conjecture to all blocks~\cite{CR} by showing that any two blocks of weight $w$ are derived equivalent. 

The blocks defined by Rouquier are now known as Rouquier blocks or RoCK blocks (the RoCK standing for Rouquier or Chuang-Kessar) and have become ubiquitous in the study of the modular representations of the symmetric groups and the Hecke algebras of type $A$. Working over a field of characteristic $p \geq 0$, there is an elegant closed formula for the decomposition numbers of the Rouquier blocks when $p=0$ or when the weight of the block is less than $p$~\cite{LeclercMiyachi2,CT,JLM}.  When the weight is greater than or equal to $p$, Turner~\cite{Turner} shows how to compute the decomposition numbers in terms of (unknown) decomposition numbers for smaller Hecke algebras. In addition to their role in the proof of Brou\'e's conjecture, the Rouquier blocks have appeared in the proofs of various results about the Hecke algebras. 

\begin{itemize}
\item The first examples of homomorphism spaces between Specht modules of dimension $>1$ over fields of characteristic $p\geq 5$ were discovered in Rouquier blocks~\cite{Dodge}.
\item The proof that if $q \neq -1$ and $\lambda$ is an $(e,p)$-irreducible partition then the Specht module $S^\la$ is irreducible was completed after the proof was reduced to the case that $\la$ is a Rouquier partition~\cite{Fayers:Irred, JLM}.  
\item The Rouquier blocks formed the base case in the proof that the decomposition numbers of blocks of weight 3 are at most 1 over fields of characteristic $p \geq 5$~\cite{Fayers:W3}. 
\item James' conjecture holds for Rouquier blocks. It was proved for blocks of weights 3 using the result above and proved for blocks of weight 4 taking the Rouquier blocks as the base case in an inductive argument~\cite{Fayers:W4}. Unfortunately, James' conjecture has since been shown to be false~\cite{Williamson}, demonstrating one of the drawbacks with the approach of proving conjectures in specific cases!   
\end{itemize}

Given these results, it is not surprising that Rouquier blocks (or RoCK blocks) have been defined for other algebras, with the definition often motivated by Brou\'e's conjecture. Examples of groups for which Rouquier blocks have been defined are the finite general linear groups~\cite{Miyachi1, Turner1}, certain other finite classical groups~\cite{Livesey}, the Chevalley groups of type $E$~\cite{Miyachi} and the double covers of symmetric groups~\cite{KleshchevLivesey}. 

In this paper we consider Rouquier blocks of the Ariki-Koike algebras $\h$. A few months after the first draft of this paper appeared on the arXiv, Webster~\cite{Webster} defined the notion of a RoCK block for any categorical module over an affine Lie algebra and dominant weight in its support. Webster's definition was inspired by the work in this paper; in the case of the Ariki-Koike algebra, his RoCK blocks form a Scopes equivalence class which contains our Rouquier blocks~\cite[Theorem~B]{Webster}. (We discuss Scopes equivalence in Section~\ref{Eqs}.) Webster's work also proved Conjecture 1 of the first version of this paper, showing that Scopes equivalence gives a Morita equivalence. Thus where the first version of our paper referred to certain blocks being `decomposition equivalent,' we have now used Webster's paper to add that they are also Morita equivalent.

The Ariki-Koike algebras were defined by Ariki and Koike~\cite{ArikiKoike} as simultaneous generalisations of the Hecke algebras of type $A$ (when $r=1$) and type $B$ (when $r=2)$. They also appear as the cyclotomic Hecke algebras of type $G(r,1,n)$~\cite{BroueMalle} and have been shown to be isomorphic to the cyclotomic Khovanov-Lauda-Rouquier algebras of type $A$~\cite{BK:Blocks}. The combinatorics which controls the representation theory of $\h$ is a generalisation of the combinatorics we see in the representation theory of the Hecke algebras $\hn=\mathcal{H}_{1,n}$. In particular, there is an important class of $\h$-modules called Specht modules which are indexed by $r$-multipartitions of $n$; when $r=1$, these Specht modules are simply indexed by partitions. All composition factors of a Specht module lie in the same block and so we talk about a multipartition (or partition) lying in a block of $\h$ (or $\hn$) if the corresponding Specht module does.  

Given a partition $\la$ and an integer $s$, we define a corresponding abacus configuration. The Rouquier blocks of $\hn$ are best described using abacuses; they are blocks in which the partitions have particular abacus configurations. It is natural to ask whether we can come up with a sensible definition of Rouquier blocks for Ariki-Koike algebras using $r$-tuples of abacus configurations to represent multipartitions.
An obvious construction would be define them so that if $\bla=(\la^{(1)},\la^{(2)},\ldots,\la^{(r)})$ lies in a Rouquier block then $\la^{(k)}$ is a Rouquier partition for $1 \leq k \leq r$, and in fact this is the definition we give; we will see that this our definition of an abacus configuration ensures that this definition does depend on the multicharge associated with $\hn$.  
However on its own, this definition is not satisfactory. For one thing, without writing down all the multipartitions in the block, it seems hard to decide whether a block is Rouquier. For another, although it is a combinatorially sound generalisation of the Rouquier blocks, we would like our blocks have interesting algebraic properties. We address these questions in the paper. 

In Sections~\ref{SS1} and \ref{Blocks}, we define the Ariki-Koike algebras and the combinatorical objects that we will use. In particular, we introduce $\A^r_e$, the set of $r$-tuples of abacus  configurations with $e$ runners, and define an equivalence relation $\ap_e$ on $\A^r_e$ which corresponds to partitioning the multipartitions into blocks. In Section~\ref{S:Uglov}, we introduce a map defined by Uglov which induces a bijection $\Psi_r:\A_e^r \rightarrow \A_e$ between $r$-tuples of abacus configurations and abacus configurations. We show that the map $\Psi_r$ preserves blocks while its inverse $\Phi_r$ splits up a block according to the $r$-residue set of the $e$-quotients of its elements. Thus we obtain a bijection between a block of $\A^r_e$ and a block of $\A^e_r$. 

We see Uglov's bijection again in Section~\ref{RouquierMP} where we define our Rouquier blocks. We show that if $\Ra$ is an $r$-Rouquier block in $\A_e$ then $\Phi_r(\Ra)$ is a union of Rouquier blocks in $\A^r_e$, thus giving us a way of generating Rouquier blocks for $\h$. Not every Rouquier block shows up as an image under $\Phi_r$, but this is not problematic as we consider equivalences of blocks. 
When $r=1$, we may use Scopes equivalence~\cite{Scopes,Jost} to show that all Rouquier blocks of the same weight are both `decomposition equivalent' and Morita equivalent. Similarly, when $r \geq 1$, Dell'Arciprete~\cite{DA} gives a combinatorial description of Scopes equivalence and shows that blocks which are Scopes equivalent are decomposition equivalent. Work of Webster~\cite{Webster} shows further that Scopes equivalent also implies Morita equivalent.  
We describe how we may `stretch' a block and show that any sufficiently stretched block is a Rouquier block with the property that its image under $\Psi_r$ lies in an $r$-Rouquier block. Moreover, if we stretch a Rouquier block we obtain a Rouquier block which is both decomposition and Morita equivalent to the original block. 

There are some obvious questions to ask about our Rouquier blocks.  We would like to know whether the RoCK blocks are Morita equivalent to some kind of `local object', analogous to the work of Chuang and Kessar. We have not given serious thought as to what such an object might be and we would welcome any results in this direction. 

A second question is whether there is a closed formula for the decomposition numbers of the Rouquier blocks, say when $p=0$ or the weight of the block is sufficiently small. 
We have given such a formula when $p=0$ and the two multipartitions have a common multicore~\cite{L:AKRCores}. 
By work of Muth, Speyer and Sutton~\cite{MSS}, these decomposition numbers should be related to decomposition numbers for cell modules in the cyclotomic wreath-zigzag algebra.  Our data when $e=r=2$ and $p=0$ indicates that in this case there is probably a closed formula but that the formula is likely to be somewhat complicated. More generally, we note that the core blocks are Scopes equivalent to Rouquier blocks and for $r \geq 3$ there is no known formula for the decomposition numbers of core blocks. 

\section{Definitions}
\subsection{The Ariki-Koike algebra} \label{SS1}
For more information on the Ariki-Koike algebras, we refer the reader to the survey paper by Mathas~\cite{Mathas:AKSurvey}. Let $r\geq 1$ and $n \geq 0$ and let $\mathbb{F}$ be a field. Choose $q \in \mathbb{F} \setminus \{0\}$ and ${\bf Q}=(Q_1,\dots,Q_r) \in \mathbb{F}^r$. The Ariki-Koike algebra $\h=\h(q,\bf Q)$ is the unital associative $\mathbb{F}$-algebra with generators $T_0 , \dots, T_{n-1}$ and relations
$$\begin{array}{crcll}
& (T_i +q )(T_i -1) & = & 0, &  \text{ for } 1 \leq i \leq n-1, \\
& T_i T_j & = & T_j T_i, & \text{ for } 0 \leq i,j \leq n-1, |i-j|>1, \\
& T_i T_{i+1} T_i & = & T_{i+1} T_i T_{i+1}, & \text{ for } 1 \leq i \leq n-2,\\
& (T_0 - Q_1)\dots (T_0 - Q_r) & = & 0, & \\
& T_0 T_1 T_0 T_1 & = & T_1 T_0 T_1 T_0. &
\end{array}$$

Define $e$ to be minimal such that $1+q+\dots+q^{e-1}=0$, or set $e=\infty$ if no such value exists. Throughout this paper we shall assume that $e$ is finite and we shall refer to $e$ as the quantum characteristic. Set $I=\{0, 1,\ldots,e-1\}$.   
We say two parameters $Q_k$ and $Q_l$ are $q$-connected if $Q_k = q^a Q_l$ for some $a \in I$.  A result of Dipper and Mathas~\cite{DM:Morita} states that each Ariki-Koike algebra~$\h$ is Morita equivalent to a direct sum of tensor products of smaller algebras whose parameters are all $q$-connected.  
In view of this result, we will further assume that all our parameters are $q$-connected, in fact, that they are all powers of $q$ where $q\neq 1$, so that $q$ is a primitive $e^{\text{th}}$ root of unity in $\mathbb{F}$.  Hence there exists a unique $\mc = (a_1,a_2,\ldots,a_r) \in I^r$ such that $Q_k = q^{a_k}$ for all $1 \leq k \leq r$. 

If $r=1$ the cyclotomic relation collapses to $T_0=Q_1 \in \mathbb{F}$ and we obtain the Hecke algebra of type $A$ which is independent of $\mc$. We shall write $\mathcal{H}_n$ for $\mathcal{H}_{1,n}$. 

The algebra $\h$ is a cellular algebra according to the definition of Graham and Lehrer~\cite{GL}. We use the cellular basis described in ~\cite{DJM:CellularBasis} so that the cell modules are indexed by $r$-multipartitions of $n$. We recall the definition of $r$-multipartitions. A partition of $n$ is a sequence $\la=(\la_1,\la_2,\ldots)$ of non-negative integers such that $\la_1 \geq \la_2 \geq \ldots$ and $\sum_{i \geq 1} \la_i =n$. We write $|\la|=n$. For convenience, we will often omit the terms equal to 0 in the partition and we will use exponent notation to gather together equal terms, for example we write $(4,4,3,2,1,1,1,0,0,\ldots) = (4^2,3,2,1^3)$. Let $\La_n$ denote the set of partitions of $n$ and $\La=\bigcup_{n \geq 0} \La_n$ denote the set of all partitions. We write $\emp$ to denote the unique partition of $0$. 

An $r$-multipartition, or multipartition, of $n$ is an $r$-tuple of partitions $\bla=(\la^{(1)},\la^{(2)},\ldots,\la^{(r)})$ such that $\sum_{k=1}^r |\la^{(k)}|=n$; we write $|\bla|=n$. Write $\Larn$ to denote the set of $r$-multipartitions of $n$ and $\Laz{r}=\bigcup_{n \geq 0} \Larn$ to denote the set of all $r$-multipartitions. 
For each multipartition $\bla \in \Larn$ we define a $\h$-module $S^{\bla}$ called a Specht module; these modules are the cell modules given by the cellular basis of $\h$~\cite{DJM:CellularBasis}. When $\h$ is semisimple, the Specht modules form a complete set of non-isomorphic irreducible $\h$-modules. However, we are mainly interested in the case when $\h$ is not semisimple. By the general theory of cellular algebras, each Specht module $S^{\bla}$ comes equipped with an $\h$-invariant bilinear form. Let $\rad(S^{\bla})$ be the radical of $S^{\bla}$ with respect to this form and set $D^{\bla}=S^{\bla}/ \rad(S^{\bla})$. Define $\Kl(\mc)=\{\bla \in \Larn \mid D^{\bla} \neq \{0\}\}$ and define $\Kla(\mc) = \bigcup_{n \geq 0} \Kl(\mc)$. Then $\{D^{\bla} \mid \bla \in \Kl(\mc)\}$ is a complete set of non-isomorphic irreducible $\h$-modules.  

If $\bla \in \Larn$ and $\bmu \in \Kl(\mc)$, let $d_{\bla\bmu}=[S^{\bla}:D^{\bmu}]$ denote the multiplicity of the simple module $D^{\bmu}$ as a composition factor of the Specht module $S^{\bla}$. The matrix $D=(d_{\bla\bmu})$ is called the decomposition matrix of $\h$ and determining its entries is one of the most important open problems in the representation theory of the Ariki-Koike algebras. Again using the general theory of cellular algebras, it is known that all of the composition factors of a Specht module $S^{\bla}$ belong to the same block; consequently we can talk about a Specht module belonging to a specific block. Whether or not $S^{\bla}$ and $S^{\bmu}$ lie in the same block depends only on the tuple $\mc \in I^r$
 and so we define an equivalence relation $\sd$ on $\La^{(r)}$ by saying that $\bla \sd \bmu$ if and only if there exists $n\geq0$ such that $|\bla|=|\bmu|=n$ and $S^{\bla}$ and $S^{\bmu}$ belong to the same block of $\h$. 

\subsection{Blocks and abacus configurations}\label{Blocks} In this section we fix $r\geq 1$. Take $q\in \mathbb{F} \setminus \{0,1\}$ such that $1+q+\ldots+q^{f-1}=0$ for some $f \in \Z$ and let $e \geq 2$ be minimal with this property. Set $I=\{0,1,\ldots,e-1\}$. If $\mc=(a_1,a_2,\ldots,a_r) \in I^r$ and $n \geq 0$ we set $\h=\h(q,{\bf Q})$ where ${\bf Q} = (q^{a_1},q^{a_2},\ldots,q^{a_r})$. 

Suppose that $\bla \in \Lar$ and $\mc \in I^r$. The Young diagram of $\bla$ is the set 
\[[\bla] = \{(x,y,k) \in \Z_{>0} \times \Z_{>0} \times \{1,2 \ldots r\} \mid y \leq \la_x^{(k)}\}.\]
If $r=1$, we may omit the third component. To each node $(x,y,k) \in [\bla]$ we associate its residue $\res_{\mc}(x,y,k) = a_k + y - x \mod e$. We draw the residue diagram of $\bla$ by replacing each node in the Young diagram by its residue.

\begin{ex} Suppose that $e=3$ and $\mc=(0,1,0)$. Let $\bla=((4,3,1),(3^2),(1^2))$ and $\bmu = ((1^3),(3^4),(1))$ so that $\bla,\bmu \in \Lambda^{(3)}_{16}$. Then the residue diagrams of $\bla$ and $\bmu$ are respectively given by 

\[\young(0120,201,1)\, \quad \young(120,012)\, \quad \young(0,2)  \qquad \text{ and } \qquad
 \young(0,2,1)\, \quad \young(120,012,201,120)\,  \quad \young(0)\,.\]  
\end{ex} 

For $\bla \in \La^{(r)}$, define the residue set of $\bla$ to be the multiset $\Res_{\mc}(\bla) = \{\res_{\mc}(\mathfrak{n}) \mid \mathfrak{n} \in [\bla]\}$. 

\begin{proposition}[\cite{LM:Blocks}, Theorem~2.11] 
Suppose that $\bla,\bmu \in \La^{(r)}$ and $\mc \in I^r$. Then $\bla \sd \bmu$ if and only if $\Res_{\mc}(\bla)=\Res_{\mc}(\bmu)$. 
\end{proposition}

\begin{ex}
Continuing the example above, we see that \[\Res_{\mc}(\bla)=\Res_{\mc}(\bmu)=\{0,0,0,0,0,0,1,1,1,1,1,2,2,2,2,2\}.\] Hence $\bla \sd \bmu$ and so $S^{\bla}$ and $S^{\bmu}$ lie in the same block of $\mathcal{H}_{3,16}(q,{\boldsymbol Q})$ where $q$ satisfies $1+q+q^2=0$ and ${\boldsymbol Q}=(q^0,q^1,q^0)=(1,q,1)$. 
\end{ex}

If $r=1$, it is obvious that the relations $\sim_{e,(a)}$ and $\sim_{e,(a')}$ agree for any $a,a' \in \Z$. In this case, we will omit the second subscript and write $\la \sim_e \mu$ if there exists $n$ such that $|\la|=|\mu|=n$ and $S^\la$ and $S^\mu$ lie in the same block of $\mathcal{H}_{n}$. For $r \geq 1$, it will be convenient later to talk about a relation $\sim_{e,\bs}$ for any $\bs \in \Z^r$. We define this relation in the natural way: If $\bs \in \Z^r$ then there is a unique $\mc \in I^r$ with $s_k \equiv a_k \mod e$ for all $1 \leq k \leq r$. Set $\sim_{e,\bs}=\sim_{e,\mc}$. 

From the definition of the equivalence relation $\sd$, it is clear that if  $\bmu \in \Kl(\mc)$ and $\bla \in \Larn$ then $d_{\bla\bmu}=0$ unless $\bmu\sd \bla$. Conversely (using the properties of a cellular basis) if $\bmu,\bla \in \Larn$ and $\bmu \sd \bla$ then there exists multipartitions $\bla=\bla^{(0)},\bla^{(1)},\ldots,\bla^{(t)}=\bmu \in \Larn$ such that for all $0 \leq z <t$ the Specht modules $S^{\bla^{(z)}}$ and $S^{\bla^{(z+1)}}$ have a common composition factor. 

The abacus was first introduced by James~\cite{James:Abacus} as a way to represent partitions. In this paper, we will use abacus configurations extensively. 
We say that an $e$-abacus configuration, or simply an abacus configuration, consists of
an abacus with $e$ runners which are infinite in both directions indexed from left to right by the elements of $I$, where the possible bead positions are indexed by the elements of $\Z$ such that bead position $b$ on the abacus is in row $m$ of runner $i$ where $b=me+i$ and $0 \leq i \leq e-1$, and where there is some row of the abacus such that all higher rows are full of beads and some row of the abacus such that all lower rows do not contain any beads.  
Let $\A_e$ denote the set of all $e$-abacus configurations. 

We say that a $\beta$-set is a subset $B \subset \mathbb{Z}$ such that for all $z \ll 0$ we have $z\in B$ and for all $z \gg 0$ we have $z \notin B$.  
Given a $\beta$-set $B$ we define the abacus configuration of $B$ to be the abacus configuration which has a bead at position $b$ for each $b \in B$. When we draw abacus configurations we will draw only a finite part of the runners and we will assume that above this point the runners are full of beads and below this point there are no beads. In our examples, we will draw a line between the strictly negative and the positive positions. 

From now on, we will identify $\beta$-sets with abacus configurations. Now suppose that $\la=(\la_1,\la_2,\ldots) \in \La$ and $s \in \mathbb{Z}$. Define the $\beta$-set \[B_s(\la)=\{\la_i - i + s \mid i \geq 1\}.\]

The following result has been used many times in the literature; we give a proof since understanding the abacus will be helpful in understanding the results in this paper. 

\begin{lemma} \label{PartBij}
Let $B$ be a $\beta$-set. Then there is a unique pair $(\la,s) \in \Lambda \times \Z$ such that $B=B_s(\la)$ and we may find this pair as follows. Suppose that $B=\{b_1,b_2,\ldots\}$ where $b_1>b_2>\ldots$. For $i \geq 1$, set $\la_i = \#\{c<b_i \mid c \notin B\}$; then $\la=(\la_1,\la_2,\ldots)$. To find $s$, repeatedly choose $b \in B$ such that there exists $c<b$ with $c \notin B$ and replace $b$ with $c$ until this is no longer possible. Then $s$ is minimal such that $s \notin B$. 
\end{lemma}

\begin{proof}
Suppose that $(\la,s) \in \Lambda \times \Z$ has the property that $B_s(\la)=B$ and that $\la=(\la_1,\la_2,\ldots,\la_l,0,0,\ldots)$ where $\la_l >0$. Then 
\[B_s(\la) = \{b_1,b_2,\ldots \ldots \} = \{\la_1+s-1,\la_2+s-2,\ldots,\la_l+s-l,s-l-1,s-l-2,\ldots\}.\]
We draw the beads on the integer line: 
\begin{center}
\begin{tikzpicture}\tikzset{yscale=0.3,xscale=0.3}
\foreach \a in {0,1,2,5,10,12,20,25} {
\node at (\a,0){$\bullet$};};
\node at (25,-1){$b_1$}; \node at (20,-1){$b_2$};
\node at (2,-1){$b_{l+1}$}; \node at (5,-1){$b_l$}; 
\node at (12,-1){$b_i$}; \node at (10,-1){$b_{i+1}$};
\foreach \a in {-1.5,16,7.5} {
\node at (\a,0){$\ldots$};}; 
\end{tikzpicture}
\end{center}
For $1 \leq i \leq l$, we have \[\#\{c<b_i \mid c \notin B\}=b_{i}-b_{l+1}-1-(l-i) = (\la_i+s-i) - (s-l-1)-1 - l+i = \la_i\]
and for $i >l$, we have $\#\{c<b_i \mid c \notin B\} = 0  = \la_{i}$. Hence $\la=(\la_1,\la_2,\ldots)$ is uniquely determined by the property that $\la_i = \{c<b_i \mid c \notin B\}$. Now bearing in mind that $s=b_{l+1}+l+1$, we can see that $s$ is indeed determined as in the lemma.

We have shown that if $(\la,s) \in \Lambda \times \Z$ is such that $B_s(\la)=B$ then $(\la,s)$ does indeed satisfy the properties above. Conversely, set $\la_i=\#\{c<b_i \mid c \notin B\}$ for all $i \geq 1$. Choose $l$ maximal such that $\la_i >0$ and set $s=b_{l+1}+l+1$. Then for all $1 \leq i \leq l$ we have
\[\la_i +s - i = \#\{c<b_i \mid c \notin B\} +b_{l+1} +l+1 - i = b_{i}-b_{l+1}-1-(l-i)+ b_{l+1}+l+1-i = b_i\]
and for $i >l$ we have $\la_i+s-i=b_{l+1}+l+1-i = b_i$. 
\end{proof}

Intuitively we can think of finding $s$ by pushing the beads on the abacus to smaller positions until there are no gaps between the beads.

\begin{corollary}\label{PartBijC}
The map sending $(\la,s) \in \La \times \Z$ to $B_s(\la) \in \A_e$ defines a bijection between $\La \times \Z$ and $\A_e$.
\end{corollary}

We now have an explicit bijection between $\La \times \mathbb{Z}$ and $\A_e$ and in future we will identify the two. 
Now suppose that $r \geq 1$. We extend the definition above to a bijection between $\La^{(r)}\times \Z^r$ and $\A_e^r$: 
\[(\la^{(1)},\la^{(2)},\ldots,\la^{(r)}) \times (s_1,s_2,\ldots,s_r) \mapsto (B_{s_1}(\la^{(1)}),B_{s_2}(\la^{(2)}),\ldots,B_{s_r}(\la^{(r)})).\]

Using the identification above, if we have an abacus configuration corresponding to $(\la,s)\in \Lambda\times\Z$ then we will simply refer to it as an abacus configuration $(\la,s)\in \A_e$; similarly if $(\bla,\bs) \in \Lambda^{(r)} \times \Z^r$ we will say that $(\bla,\bs) \in \A_e^r$. 

We define an equivalence relation on $\A_e^r$ by saying that $(\bla,\bs) \ap_e (\bmu,\bs')$ if and only if $\bs=\bs'$ and $\bla \sim_{e,\bs} \bmu$. Let $\R$ be a $\ap_e$-equivalence class of $\A_e^r$. Then there exist $n \geq 0$ and $\bs \in \Z^r$ such that each element of $\R$ has the form $(\bla,\bs)$ where $\bla \in \Lambda^{(r)}_n$. Let ${\bf Q}=(q^{s_1},q^{s_2},\ldots,q^{s_r})$ and let $\mathcal{H}=\mathcal{H}_{r,n}(q,{\bf Q})$. Then the Specht modules $S^{\bla}$ for $(\bla,\bs) \in \R$ all belong to the same block of $\mathcal{H}$. We will denote this block by $\widetilde{\R}$. We will also refer to the $\ap_e$-equivalence classes as blocks.

\begin{ex} As before, let $e=3$ and $\bla=((4,3,1),(3,3),(1,1)), \bmu=((1,1,1),(3,3,3,3),(1))$. Let $\bs=(0,1,0)$. Then the abacus configurations of $(\bla,\bs)$ and $(\bmu,\bs)$ are respectively given by  
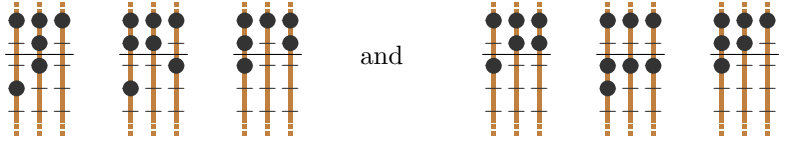
\begin{center}
\begin{tikzpicture}\tikzset{yscale=0.3,xscale=0.3}
\foreach \k in {0,1,2,5,6,7,10,11,12} {
\fill [color=brown] (\k-.1,2)--(\k-.1,7.0) -- (\k+0.1,7.0)-- (\k+.1,2)--(\k-.1,2);
\fill[color=brown](\k-.1,7.1)--(\k+.1,7.1)--(\k+.1,7.3)--(\k-.1,7.3)--(\k-.1,7.1);
\fill[color=brown](\k-.1,1.9)--(\k+.1,1.9)--(\k+.1,1.7)--(\k-.1,1.7)--(\k-.1,1.9);
\fill[color=brown](\k-.1,1.6)--(\k+.1,1.6)--(\k+.1,1.4)--(\k-.1,1.4)--(\k-.1,1.6);
}
\foreach \k in {3.5,6.5} {
\fill[color=bead] (0,\k) circle (10pt);}
\foreach \k in {4.5,5.5,6.5} {
\fill[color=bead] (1,\k) circle (10pt);}
\foreach \k in {6.5} {
\fill[color=bead] (2,\k) circle (10pt);}
\foreach \k in {3.5,5.5,6.5} {
\fill[color=bead] (5,\k) circle (10pt);}
\foreach \k in {5.5,6.5} {
\fill[color=bead] (6,\k) circle (10pt);}
\foreach \k in {4.5,6.5} {
\fill[color=bead] (7,\k) circle (10pt);}
\foreach \k in {4.5,5.5,6.5} {
\fill[color=bead] (10,\k) circle (10pt);}
\foreach \k in {6.5} {
\fill[color=bead] (11,\k) circle (10pt);}
\foreach \k in {5.5,6.5} {
\fill[color=bead] (12,\k) circle (10pt);}
\draw(-0.5,5) -- (2.5,5); \draw(4.5,5)--(7.5,5); \draw(9.5,5)--(12.5,5); 
\node at (16,5){and};
\foreach \a in {0,1,2,5,6,7,10,11,12} {
\foreach \k in {2.5,3.5,4.5,5.5,6.5}{
\node at (\a,\k)[color=bead]{$-$};};};
\end{tikzpicture} \qquad
\begin{tikzpicture}\tikzset{yscale=0.3,xscale=0.3}
\foreach \k in {0,1,2,5,6,7,10,11,12} {
\fill [color=brown] (\k-.1,2)--(\k-.1,7.0) -- (\k+0.1,7.0)-- (\k+.1,2)--(\k-.1,2);
\fill[color=brown](\k-.1,7.1)--(\k+.1,7.1)--(\k+.1,7.3)--(\k-.1,7.3)--(\k-.1,7.1);
\fill[color=brown](\k-.1,1.9)--(\k+.1,1.9)--(\k+.1,1.7)--(\k-.1,1.7)--(\k-.1,1.9);
\fill[color=brown](\k-.1,1.6)--(\k+.1,1.6)--(\k+.1,1.4)--(\k-.1,1.4)--(\k-.1,1.6);
}
\foreach \k in {4.5,6.5} {
\fill[color=bead] (0,\k) circle (10pt);}
\foreach \k in {5.5,6.5} {
\fill[color=bead] (1,\k) circle (10pt);}
\foreach \k in {5.5,6.5} {
\fill[color=bead] (2,\k) circle (10pt);}
\foreach \k in {3.5,4.5,6.5} {
\fill[color=bead] (5,\k) circle (10pt);}
\foreach \k in {4.5,6.5} {
\fill[color=bead] (6,\k) circle (10pt);}
\foreach \k in {4.5,6.5} {
\fill[color=bead] (7,\k) circle (10pt);}
\foreach \k in {4.5,5.5,6.5} {
\fill[color=bead] (10,\k) circle (10pt);}
\foreach \k in {5.5,6.5} {
\fill[color=bead] (11,\k) circle (10pt);}
\foreach \k in {6.5} {
\fill[color=bead] (12,\k) circle (10pt);}
\draw(-0.5,5) -- (2.5,5); \draw(4.5,5)--(7.5,5); \draw(9.5,5)--(12.5,5); 
\foreach \a in {0,1,2,5,6,7,10,11,12} {
\foreach \k in {2.5,3.5,4.5,5.5,6.5}{
\node at (\a,\k)[color=bead]{$-$};};};
\node at (13,5){.};
\end{tikzpicture}
\end{center}
We have seen previously that $\bla \sim_{e,\bs} \bmu$ so $(\bla,\bs) \ap_e (\bmu,\bs)$. 
\end{ex}

We would now like to describe the relation $\ap_e$ in terms of abacus configurations. If $\la \in \La$, define the rim of $[\la]$ to be the set of nodes $\{(i,j) \in [\la] \mid (i+1,j+1) \notin [\la]\}$. Let $h >0$. Then a $h$-rim hook of $[\la]$ is a connected set of $h$ nodes from the rim of $[\la]$ such that removing these nodes from $[\la]$ leaves the Young diagram of a partition.

\begin{lemma}
Suppose that $(\la,s) \in \A_e$ corresponds to the $\beta$-set $B$. Suppose that $b,c \in \Z$ are such that $c<b$ and $b \in B$ and $c \notin B$. Let $B'$ be the $\beta$-set obtained by replacing $b$ with $c$. Then $B'$ corresponds to $(\mu,s) \in \A_e$ where $[\mu]$ is obtained from $[\la]$ by removing one $(b-c)$-rim hook. 
\end{lemma}

\begin{proof}
It is well-known that decreasing a $\beta$-number by $h$ corresponds to removing a $h$-rim hook from $[\la]$. It follows from Lemma~\ref{PartBij} that moving a bead in an abacus configuration does not change the value of $s$. 
\end{proof}

Suppose $\la \in \Lambda$. Choose $s \in\Z$ and draw the abacus configuration for $(\la,s)$. Define the $e$-core $\bar{\la}$ of $\la$ so that $(\bar{\la},s)$ is the abacus configuration obtained by repeated taking a bead on the abacus with a gap immediately above it and moving the bead into this gap until this is no longer possible; this corresponds to removing $e$-rim hooks from $\la$. (In fact, the $e$-core is independent of the choice of $s$ since changing $s$ only shifts the residues in the Young diagram, but we would like to use the abacus to describe it and so we must pick an $s$.) We define the weight of $\la$, $\wt(\la)$, to be the number of times we moved a bead up one position to get from $\la$ to $\bar{\la}$; equivalently $\wt(\la) = (|\la|-|\bar{\la}|)/e$. 
We can now state the (so-called) Nakayama Conjecture.

\begin{proposition} [\cite{Brauer,Robinson}] \label{Whenr1}
Suppose that $\la,\mu \in \La$. Then $\la \sim_e \mu$ if and only if $\wt(\la) = \wt(\mu)$ and $\bar{\la}=\bar{\mu}$.
\end{proposition}

In order to describe the $\approx_e$-equivalence classes of $\A_e$, we introduce a little more notation. Suppose that $(\la,s) \in \A_e$ corresponds to a $\beta$-set $B$. For $i \in I$, let \[B_i = \{b \in B \mid b \equiv i \mod e\}, \qquad \qquad C_i = \{(b-i)/e \mid b \in B_i\}.\] Then each $C_i$ is a $\beta$-set and hence corresponds to a pair $(\rho_i,t_i) \in \La \times \Z$. Define a map $\eta:\A_e \rightarrow \La^{(e)} \times \Z^e$ by setting $\eta(\la,s) = (\brho, \bt) = ((\rho_0,\rho_1,\ldots,\rho_{e-1}),(t_0,t_1,\ldots,t_{e-1}))$. 

In James's notation, $\brho$ is the $e$-quotient of $\la$ and $\bt$ determines the $e$-core. Note that knowledge of $(\la,s)$ is equivalent to knowledge of $\eta(\la,s)$; also that $\wt(\la)=|\brho|$. 

\begin{lemma}
Suppose that $(\la,s),(\mu,s') \in \A_e$ are such that $\eta(\la,s)=(\brho,\bt)$ and $\eta(\mu,s')=(\brho',\bt')$. Then $(\la,s) \ap_e (\mu,s')$ if and only if $\bt=\bt'$ and $|\brho|=|\brho'|$. 
\end{lemma}

\begin{proof}
We have $(\la,s) \ap_e (\mu,s')$ if and only if $s=s'$ and $\la\sim_e \mu$; so by Proposition~\ref{Whenr1}, $(\la,s) \ap_e (\mu,s')$ if and only if $s=s'$, $\wt(\la)=\wt(\mu)$ and $\bar{\la}=\bar{\mu}$. Now if $s=s'$ then
\[\wt(\la)=\wt(\mu) \iff |\brho|=|\brho'| \quad \text{and } \quad \bar{\la}=\bar{\mu} \iff \bt=\bt'.\]
\end{proof}

\begin{ex}
  Let $s=6$ and set $\la=(4,2^3)$ and $\mu=(4,3,2,1)$. Then the abacus configurations of $(\la,s)$ and $(\mu,s)$ are respectively given by
  
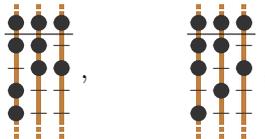
\begin{center}
\begin{tikzpicture}\tikzset{yscale=0.3,xscale=0.3}
\foreach \k in {0,1,2,8,9,10} {
\fill [color=brown] (\k-.1,2)--(\k-.1,7.0) -- (\k+0.1,7.0)-- (\k+.1,2)--(\k-.1,2);
\fill[color=brown](\k-.1,7.1)--(\k+.1,7.1)--(\k+.1,7.3)--(\k-.1,7.3)--(\k-.1,7.1);
\fill[color=brown](\k-.1,1.9)--(\k+.1,1.9)--(\k+.1,1.7)--(\k-.1,1.7)--(\k-.1,1.9);
\fill[color=brown](\k-.1,1.6)--(\k+.1,1.6)--(\k+.1,1.4)--(\k-.1,1.4)--(\k-.1,1.6);
}
\foreach \k in {2.5,3.5,5.5,6.5} {
\fill[color=bead] (0,\k) circle (10pt);}
\foreach \k in {4.5,5.5,6.5} {
\fill[color=bead] (1,\k) circle (10pt);}
\foreach \k in {6.5,4.5} {
\fill[color=bead] (2,\k) circle (10pt);}
\foreach \k in {2.5,4.5,5.5,6.5} {
\fill[color=bead] (8,\k) circle (10pt);}
\foreach \k in {3.5,5.5,6.5} {
\fill[color=bead] (9,\k) circle (10pt);}
\foreach \k in {4.5,6.5} {
\fill[color=bead] (10,\k) circle (10pt);}
\draw(-.5,6) -- (2.5,6); \draw(7.5,6)--(10.5,6);  
\foreach \a in {0,1,2,8,9,10} {
\foreach \k in {2.5,3.5,4.5,5.5,6.5}{
\node at (\a,\k)[color=bead]{$-$};};};
\node at (3,4){,}; \node at (11,4){,}; 
\end{tikzpicture} 
\end{center}
so that
\begin{align*}
\eta(\la,s) & = (((1^2),\emp,(1)),(3,2,1)), & \eta(\mu,s) & = (((1),(1),(1)),(3,2,1)),
\end{align*}
and $(\la,s) \ap_e (\mu,s)$. The mutual core $\bar{\la}=\bar{\mu}$ can be determined by noting that $\eta(\bar{\la},s)=\eta(\bar{\mu},s) = ((\emp,\emp,\emp),(3,2,1))$ and has abacus configuration
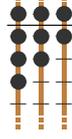
\begin{center}
\begin{tikzpicture}\tikzset{yscale=0.3,xscale=0.3}
\foreach \k in {0,1,2} {
\fill [color=brown] (\k-.1,2)--(\k-.1,7.0) -- (\k+0.1,7.0)-- (\k+.1,2)--(\k-.1,2);
\fill[color=brown](\k-.1,7.1)--(\k+.1,7.1)--(\k+.1,7.3)--(\k-.1,7.3)--(\k-.1,7.1);
\fill[color=brown](\k-.1,1.9)--(\k+.1,1.9)--(\k+.1,1.7)--(\k-.1,1.7)--(\k-.1,1.9);
\fill[color=brown](\k-.1,1.6)--(\k+.1,1.6)--(\k+.1,1.4)--(\k-.1,1.4)--(\k-.1,1.6);
}
\foreach \k in {3.5,4.5,5.5,6.5} {
\fill[color=bead] (0,\k) circle (10pt);}
\foreach \k in {4.5,5.5,6.5} {
\fill[color=bead] (1,\k) circle (10pt);}
\foreach \k in {6.5,5.5} {
\fill[color=bead] (2,\k) circle (10pt);}
\draw(-.5,6) -- (2.5,6); 
\foreach \a in {0,1,2} {
\foreach \k in {2.5,3.5,4.5,5.5,6.5}{
\node at (\a,\k)[color=bead]{$-$};};};
\end{tikzpicture} 
\end{center}
obtained by pushing the beads up on the abacus configuration of $(\la,s)$ (and $(\mu,s)$). Thus $\bar{\la}=\bar{\mu}=(1)$. 
\end{ex}

 If $\bt=(t_0,t_1,\ldots,t_{e-1}) \in \Z^e$ and $w \geq 0$, define the $\ap_e$-equivalence class
\[\Ra(t_0,t_1,\ldots,t_{e-1};w)=\Ra(\bt;w) = \Bigg\{(\la,s) \in \A_e \mid \eta(\la,s) = (\brho,\bt) \text{ and } |\brho| = w\Bigg\}.\]

When $r>1$ it is more complicated to describe the relation $\ap_e$ using abacus configurations but it will be useful to be able to do so. Let $(\bla,\bs),(\bmu,\bs) \in \A^r_e$. 
\begin{enumerate}
\item We say that $(\bla,\bs) \xrightarrow{e}_1 (\bmu,\bs)$ if we can obtain $(\bmu,\bs)$ from $(\bla,\bs)$ by decreasing a $\beta$-number by $e$ in component $k_1$ and increasing a $\beta$-number by $e$ in component $k_2$ for some $1 \leq k_1,k_2 \leq r$. 
\item We say that $(\bla,\bs) \xrightarrow{e}_{2} (\bmu,\bs)$ if we can obtain $(\bmu,\bs)$ from $(\bla,\bs)$ as follows. 
Suppose that there exist $1 \leq k_1,k_2 \leq r$, $b_1,b_2 \in \Z$ and $h > 0$ such that 
\begin{itemize}
\item $b_1 \equiv b_2 \mod e$;
\item $b_1 \in B_{s_{k_1}}(\la^{(k_1)})$ and $b_1+h \notin B_{s_{k_1}}(\la^{(k_1)})$;
\item $b_2 \notin B_{s_{k_2}}(\la^{(k_2)})$ and $b_2+h \in B_{s_{k_2}}(\la^{(k_2)})$.
\end{itemize}
Form $(\bmu,\bs) \in \A_e$ to be the configuration where the $\beta$-numbers in each component agree with those of $(\bla,\bs)$, except that we replace $b_1$ in $B_{s_{k_1}}(\la^{(k_1)})$ with $b_1+h$ and we replace $b_2+h$ in $B_{s_{k_2}}(\la^{(k_2)})$ with $b_2$.
\end{enumerate}
Say that $(\bla,\bs) \xrightarrow{e} (\bmu,\bs)$ if $(\bla,\bs) \xrightarrow{e}_{z} (\bmu,\bs)$ for $z\in \{1,2\}$.

Note that the operation described by the relation $(\bla,\bs) \xrightarrow{e}_2 (\bmu,\bs)$ corresponds to removing a $h$-rim hook from component $k_2$ of $[\bla]$ and adding a $h$-rim hook to component $k_1$ of $[\bla]$. The condition $b_1 \equiv b_2 \mod e$ ensures that if $\mc \in I^r$ with $a_i \equiv s_i$ for $0 \leq i \leq e-1$ then $\Res_{\mc}(\bmu)=\Res_{\mc}(\bla)$. We may see this in the example below.

\begin{ex} Take $(\bla,\bs)$ to be the abacus configuration on the left and $(\bmu,\bs)$ the abacus configuration on the right. Take $k_1=2$ and $k_2=1$. Let $b_1=-1, b_2 =3$ and $h=1$.  Then $(\bla,\bs) \xrightarrow{4}_2 (\bmu,\bs)$. 

\begin{center}
\begin{tikzpicture}\tikzset{yscale=0.3,xscale=0.3}
\foreach \k in {-1,0,1,2,5,6,7,8} {
\fill [color=brown] (\k-.1,2)--(\k-.1,7.0) -- (\k+0.1,7.0)-- (\k+.1,2)--(\k-.1,2);
\fill[color=brown](\k-.1,7.1)--(\k+.1,7.1)--(\k+.1,7.3)--(\k-.1,7.3)--(\k-.1,7.1);
\fill[color=brown](\k-.1,1.9)--(\k+.1,1.9)--(\k+.1,1.7)--(\k-.1,1.7)--(\k-.1,1.9);
\fill[color=brown](\k-.1,1.6)--(\k+.1,1.6)--(\k+.1,1.4)--(\k-.1,1.4)--(\k-.1,1.6);
}
\foreach \a in {-1,0,1,2,5,6,7,8} {
\foreach \k in {2.5,3.5,4.5,5.5,6.5}{
\node at (\a,\k)[color=bead]{$-$};};};
\foreach \k in {3.5,6.5} {
\fill[color=bead] (-1,\k) circle (10pt);}
\foreach \k in {6.5} {
\fill[color=bead] (0,\k) circle (10pt);}
\foreach \k in {4.5,5.5,6.5} {
\fill[color=bead] (1,\k) circle (10pt);}
\foreach \k in {6.5} {
\fill[color=bead] (2,\k) circle (10pt);}
\foreach \k in {3.5,5.5,6.5} {
\fill[color=bead] (5,\k) circle (10pt);}
\foreach \k in {5.5,6.5} {
\fill[color=bead] (6,\k) circle (10pt);}
\foreach \k in {4.5,6.5} {
\fill[color=bead] (7,\k) circle (10pt);}
\foreach \k in {3.5,6.5} {
\fill[color=bead] (8,\k) circle (10pt);}
\draw(-1.5,5) -- (2.5,5); \draw(4.5,5)--(8.5,5); 
\fill[color=red] (8,5.5) circle (10pt); 
\fill[color=red] (-1,3.5) circle (10pt);
\end{tikzpicture} 
\raisebox{0.7cm}{$ \xrightarrow{4}_2$}
\begin{tikzpicture}\tikzset{yscale=0.3,xscale=0.3}
\foreach \k in {-1,0,1,2,5,6,7,8} {
\fill [color=brown] (\k-.1,2)--(\k-.1,7.0) -- (\k+0.1,7.0)-- (\k+.1,2)--(\k-.1,2);
\fill[color=brown](\k-.1,7.1)--(\k+.1,7.1)--(\k+.1,7.3)--(\k-.1,7.3)--(\k-.1,7.1);
\fill[color=brown](\k-.1,1.9)--(\k+.1,1.9)--(\k+.1,1.7)--(\k-.1,1.7)--(\k-.1,1.9);
\fill[color=brown](\k-.1,1.6)--(\k+.1,1.6)--(\k+.1,1.4)--(\k-.1,1.4)--(\k-.1,1.6);
}
\foreach \a in {-1,0,1,2,5,6,7,8} {
\foreach \k in {2.5,3.5,4.5,5.5,6.5}{
\node at (\a,\k)[color=bead]{$-$};};};
\foreach \k in {6.5} {
\fill[color=bead] (-1,\k) circle (10pt);}
\foreach \k in {6.5} {
\fill[color=bead] (0,\k) circle (10pt);}
\foreach \k in {4.5,5.5,6.5} {
\fill[color=bead] (1,\k) circle (10pt);}
\foreach \k in {6.5} {
\fill[color=bead] (2,\k) circle (10pt);}
\foreach \k in {3.5,5.5,6.5} {
\fill[color=bead] (5,\k) circle (10pt);}
\foreach \k in {5.5,6.5} {
\fill[color=bead] (6,\k) circle (10pt);}
\foreach \k in {4.5,6.5} {
\fill[color=bead] (7,\k) circle (10pt);}
\foreach \k in {3.5,6.5} {
\fill[color=bead] (8,\k) circle (10pt);}
\draw(-1.5,5) -- (2.5,5); \draw(4.5,5)--(8.5,5); 
\fill[color=red] (5,4.5) circle (10pt);
\fill[color=red] (2,4.5) circle (10pt);  
\end{tikzpicture} 

\end{center}
\end{ex}

\begin{proposition}[\cite{Fayers:Cores} Proposition~3.7] \label{GenBlock}
Let $\bla,\bmu \in \Lar$ and $\bs \in \Z^r$. Then $(\bla,\bs)\ap_e (\bmu,\bs)$ if and only if there exists a sequence
  \[(\bla,\bs) = (\bla_0,\bs) \xrightarrow{e} (\bla_1,\bs) \xrightarrow{e} \ldots \xrightarrow{e} (\bla_t,\bs)=(\bmu,\bs).\]
\end{proposition}

Proposition~\ref{GenBlock} shows that we may generate the $\sd$-equivalence class of $\bla$ (that is, find the Specht modules $S^{\bmu}$ which lie in the same block as $S^{\bla}$) by repeatedly adding and removing rim hooks from $[\bla]$ such that $\Res_{\mc}(\bla)$ is preserved at each stage. In practical terms, this can be slow. In the next section we use a map defined by Uglov~\cite{Uglov} which can sometimes make this process much quicker. We will also use Uglov's map in Section~\ref{RouquierMP}. 

\subsection{Uglov's map} \label{S:Uglov}
In this section we continue with the assumptions from Section~\ref{Blocks}, namely that $r$, $q$ and $e$ are fixed, and given $n \geq 0$ and $\mc \in I^r$ we define an algebra $\h$.  
For $1 \leq k\leq r$, we define a map $\psi_k: \Z \rightarrow \Z$ as follows. For $x\in \Z$, write $x=me+i$ where $0 \leq i <e$. Then $\psi_k(x) = ((m+1)r-k)e+i$. 

The next result follows immediately from this definition. 

\begin{lemma} \label{PsiMove}
Suppose $1 \leq k \leq r$ and $x \in \Z$; write $x=me+i$ where $0 \leq i <e$. Suppose $h \in \Z$ with $h>0$ and write $h=m'e+i'$ where $0 \leq i'<e$. Then
\[\psi_k(x+h) = \begin{cases}
\psi_k(x) + m're + i', & i+i' <e, \\
\psi_k(x) + m're + i' + (r-1)e, & i+i' \geq e.
\end{cases}\]
\end{lemma}

Now let $(\bla,\bs) \in \A_e^r$. For each $1 \leq k \leq r$ we have the $\beta$-set $B_{s_k}(\la^{(k)})$. Set $B=\bigsqcup_{k=1}^r \psi_k(B_{s_k}(\la^{(k)}))$; note that this is a disjoint union since if $k \neq k'$ then $\im \psi_k \cap \im \psi_{k'} = \emptyset$. 
If $z\in\Z$ then there exist unique integers $m,\,k$ and $i$ such that $z=((m+1)r-k)e+i$ and if $z$ is sufficiently small then, since $B_{s_k}(\la^{(k)})$ is a $\beta$-set, we have $me+i \in B_{s_k}(\la^{(k)})$ so that $z \in B$; similarly if $z$ is sufficiently large then $me+i \notin B_{s_k}(\la^{(k)})$ and so $z \notin B$. Thus $B$ itself is a $\beta$-set and so corresponds to $(\tilde{\la},\tilde{s}) \in \A_e$. We define $\Psi_r:\A_e^r \rightarrow \A_e$ by setting $\Psi_r(\bla,\bs)=(\tilde{\la},\tilde{s})$.

The map $\Psi_r$ was originally defined by Uglov~\cite[\S~ 4.1]{Uglov}. It appears in~\cite{JL}, and indeed, our understanding of the map came from this paper. The easiest way to understand it (in our opinion) is to look at the abacus configurations. Draw the $e$-abacus configurations for $(\bla,\bs)$. Create a new $e$-abacus by going from right to left and from top to bottom across these configurations, repeatedly taking a row of an abacus and adding it to the new abacus. By considering the map $\Psi_r$ in this way, it is clear that it is a bijection from $\A_e^r$ to $\A_e$. We let $\Phi_r = \Psi_r^{-1}$. 

If $r=1$ then $\Psi_r$ and $\Phi_r$ are simply the identity maps. 

\begin{ex}
We can see this procedure in the example below. Each $3 \times 3$ block on the right corresponds to 3 rows of length 3 on the left. 

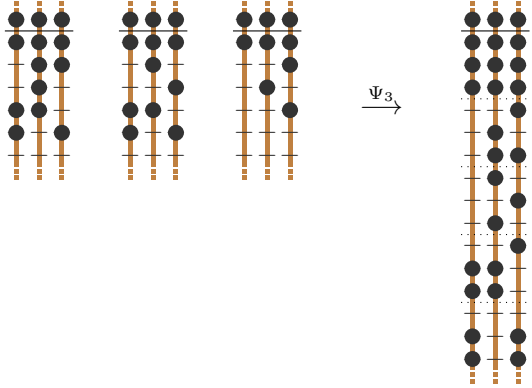
\begin{center}
\begin{tikzpicture}\tikzset{yscale=0.3,xscale=0.3}
\foreach \k in {0,1,2,5,6,7,10,11,12} {
\fill [color=brown] (\k-.1,0)--(\k-.1,7) -- (\k+0.1,7)-- (\k+.1,0)--(\k-.1,0);
\fill[color=brown](\k-.1,7.1)--(\k+.1,7.1)--(\k+.1,7.3)--(\k-.1,7.3)--(\k-.1,7.1);
\fill[color=brown](\k-.1,-.1)--(\k+.1,-.1)--(\k+.1,-.3)--(\k-.1,-.3)--(\k-.1,-.1);
\fill[color=brown](\k-.1,-.4)--(\k+.1,-.4)--(\k+.1,-.6)--(\k-.1,-.6)--(\k-.1,-.4);
}
\foreach \k in {1.5,2.5,5.5,6.5} {
\fill[color=bead] (0,\k) circle (10pt);}
\foreach \k in {2.5,3.5,4.5,5.5,6.5} {
\fill[color=bead] (1,\k) circle (10pt);}
\foreach \k in {1.5,4.5,5.5,6.5} {
\fill[color=bead] (2,\k) circle (10pt);}
\foreach \k in {1.5,2.5,5.5,6.5} {
\fill[color=bead] (5,\k) circle (10pt);}
\foreach \k in {2.5,4.5,5.5,6.5} {
\fill[color=bead] (6,\k) circle (10pt);}
\foreach \k in {1.5,3.5,5.5,6.5} {
\fill[color=bead] (7,\k) circle (10pt);}
\foreach \k in {5.5,6.5} {
\fill[color=bead] (10,\k) circle (10pt);}
\foreach \k in {3.5,5.5,6.5} {
\fill[color=bead] (11,\k) circle (10pt);}
\foreach \k in {2.5,4.5,5.5,6.5} {
\fill[color=bead] (12,\k) circle (10pt);}
\draw(-.5,6) -- (2.5,6); \draw(4.5,6)--(7.5,6); \draw(9.5,6)--(12.5,6); 
\foreach \a in {0,1,2,5,6,7,10,11,12} {
\foreach \k in {0.5,1.5,2.5,3.5,4.5,5.5,6.5}{
\node at (\a,\k)[color=bead]{$-$};};};
\foreach \k in {20,21,22} {
\fill [color=brown] (\k-.1,-9)--(\k-.1,7) -- (\k+0.1,7)-- (\k+.1,-9)--(\k-.1,-9);
\fill[color=brown](\k-.1,7.1)--(\k+.1,7.1)--(\k+.1,7.3)--(\k-.1,7.3)--(\k-.1,7.1);
\fill[color=brown](\k-.1,-9.1)--(\k+.1,-9.1)--(\k+.1,-9.3)--(\k-.1,-9.3)--(\k-.1,-9.1);
\fill[color=brown](\k-.1,-9.4)--(\k+.1,-9.4)--(\k+.1,-9.6)--(\k-.1,-9.6)--(\k-.1,-9.4);
}
\foreach \k in {-8.5,-7.5,-5.5,-4.5,3.5,4.5,5.5,6.5} {
\fill[color=bead] (20,\k) circle (10pt);}
\foreach \k in {-5.5,-4.5,-2.5,-0.5,0.5,1.5,3.5,4.5,5.5,6.5} {
\fill[color=bead] (21,\k) circle (10pt);}
\foreach \k in {-8.5,-7.5,-3.5,-1.5,0.5,2.5,3.5,4.5,5.5,6.5} {
\fill[color=bead] (22,\k) circle (10pt);}
\draw(19.5,6) -- (22.5,6);
\draw[dotted](19.5,3)--(22.5,3);
\draw[dotted](19.5,0)--(22.5,0);
\draw[dotted](19.5,-3)--(22.5,-3);
\draw[dotted](19.5,-6)--(22.5,-6);
\foreach \a in {20,21,22} {
\foreach \k in {-8.5,...,6.5}{
\node at (\a,\k)[color=bead]{$-$};};};
\node at (16,3){$\overset{\Psi_3}{\longrightarrow}$};

\end{tikzpicture}
\end{center}
\end{ex}

Jacon and Lecouvey~\cite[Corollary~2.27]{JL} show that if we make certain assumptions on $\bs$, the maps $\Psi_r$ and $\Phi_r$ both preserve blocks. Without their assumptions, one direction of this result still holds. 

\begin{proposition}\label{BlockPreserve}
Let $\bs \in \Z^r$. Suppose that $\bla,\bmu \in \Lambda^{(r)}$ and $(\bla,\bs) \ap_e (\bmu,\bs)$. Then $\Psi_r(\bla,\bs) \ap_e \Psi_r(\bmu,\bs)$.
\end{proposition}

\begin{proof}
Using Proposition~\ref{GenBlock}, it is sufficient to show that the result is true if $(\bla,\bs) \xrightarrow{e}_1 (\bmu,\bs)$ or $(\bla,\bs) \xrightarrow{e}_2 (\bmu,\bs)$. Suppose that $(\bla,\bs) \xrightarrow{e}_z (\bmu,\bs)$ for $z \in \{1,2\}$. Since we obtain $(\bmu,\bs)$ from $(\bla,\bs)$ by moving beads around in the abacus configurations, Lemma~\ref{PartBij} tells us that $\Psi_r(\bla,\bs)=(\la,s)$ and $\Psi_r(\bmu,\bs)=(\mu,s)$ for some $\la,\mu \in \La$ and $s \in \Z$. 
It remains to show that $\la \sim_e \mu$. 

Suppose $(\bla,\bs) \xrightarrow{e}_1 (\bmu,\bs)$. Note that by Lemma~\ref{PsiMove}, if $b \in \Z$ then for $1 \leq k \leq r$ we have $\psi_k(b + e) = \psi_k(b) + re$. So if we form $\bmu$ by removing an $e$-rim hook from component $k_1$ of $\bla$ and adding an $e$-rim hook to component $k_2$ of $\bla$ then we form $(\mu,s)$ from $(\la,s)$ by removing $r$ $e$-rim hooks and adding $r$ $e$-rim hooks. Hence $\la \sim_e \mu$.

Now suppose $(\bla,\bs) \xrightarrow{e}_2 (\bmu,\bs)$, where we use the notation given in the definition. Suppose $b_1=me+i$ and $h=m'e+i'$. First suppose that $i+i'<e$. Then we form $(\mu,s)$ from $(\la,s)$ by first moving a bead up $m're+i'$ positions so that it moves from runner $i$ to runner $i+i'$ and then moving a bead down $m're+i'$ positions so that it moves from runner $i+i'$ to runner $i$. Since we move the beads the same distance, $\la$ and $\mu$ are partitions of the same integer and since we end up with the number of beads on each runner unchanged, they have the same $e$-core. 

The argument is similar if $i+i'\geq e$. 
\end{proof}

The natural converse to Proposition~\ref{BlockPreserve} is false as we see below.

\begin{ex}
Let $s=16$ and suppose that $(\la,s)$ and $(\mu,s)$ have abacus configurations given respectively by
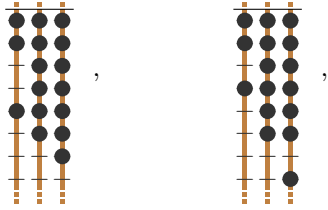
\begin{center}
\begin{tikzpicture}\tikzset{yscale=0.3,xscale=0.3}
\foreach \k in {0,1,2} {
\fill [color=brown] (\k-.1,-1)--(\k-.1,7) -- (\k+0.1,7)-- (\k+.1,-1)--(\k-.1,-1);
\fill[color=brown](\k-.1,7.1)--(\k+.1,7.1)--(\k+.1,7.3)--(\k-.1,7.3)--(\k-.1,7.1);
\fill[color=brown](\k-.1,-1.1)--(\k+.1,-1.1)--(\k+.1,-1.3)--(\k-.1,-1.3)--(\k-.1,-1.3);
\fill[color=brown](\k-.1,-1.4)--(\k+.1,-1.4)--(\k+.1,-1.6)--(\k-.1,-1.6)--(\k-.1,-1.4);
}
\foreach \k in {2.5,5.5,6.5} {
\fill[color=bead] (0,\k) circle (10pt);}
\foreach \k in {1.5,2.5,3.5,4.5,5.5,6.5} {
\fill[color=bead] (1,\k) circle (10pt);}
\foreach \k in {0.5,1.5,2.5,3.5,4.5,5.5,6.5} {
\fill[color=bead] (2,\k) circle (10pt);}
\draw(-.5,7) -- (2.5,7);
\node at (3.5,4) {,};
\foreach \k in {10,11,12} {
\fill [color=brown] (\k-.1,-1)--(\k-.1,7) -- (\k+0.1,7)-- (\k+.1,-1)--(\k-.1,-1);
\fill[color=brown](\k-.1,7.1)--(\k+.1,7.1)--(\k+.1,7.3)--(\k-.1,7.3)--(\k-.1,7.1);
\fill[color=brown](\k-.1,-1.1)--(\k+.1,-1.1)--(\k+.1,-1.3)--(\k-.1,-1.3)--(\k-.1,-1.3);
\fill[color=brown](\k-.1,-1.4)--(\k+.1,-1.4)--(\k+.1,-1.6)--(\k-.1,-1.6)--(\k-.1,-1.4);
}
\foreach \k in {3.5,5.5,6.5} {
\fill[color=bead] (10,\k) circle (10pt);}
\foreach \k in {1.5,2.5,3.5,4.5,5.5,6.5} {
\fill[color=bead] (11,\k) circle (10pt);}
\foreach \k in {-0.5,1.5,2.5,3.5,4.5,5.5,6.5} {
\fill[color=bead] (12,\k) circle (10pt);}
\draw(9.5,7) -- (12.5,7);
\node at (13.5,4){,}; 
\foreach \a in {0,1,2,10,11,12} {
\foreach \k in {-0.5,...,6.5}{
\node at (\a,\k)[color=bead]{$-$};};};
\end{tikzpicture}
\end{center}
so that $\la \sim_e \mu$. Then their images under $\Phi_2$ are
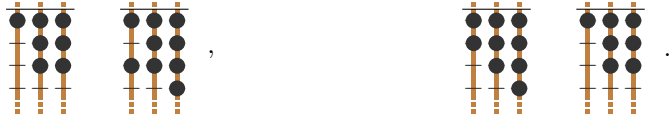
\begin{center}
\begin{tikzpicture}\tikzset{yscale=0.3,xscale=0.3}
\foreach \k in {0,1,2,5,6,7} {
\fill [color=brown] (\k-.1,3)--(\k-.1,7) -- (\k+0.1,7)-- (\k+.1,3)--(\k-.1,3);
\fill[color=brown](\k-.1,7.1)--(\k+.1,7.1)--(\k+.1,7.3)--(\k-.1,7.3)--(\k-.1,7.1);
\fill[color=brown](\k-.1,2.9)--(\k+.1,2.9)--(\k+.1,2.7)--(\k-.1,2.7)--(\k-.1,2.9);
\fill[color=brown](\k-.1,2.6)--(\k+.1,2.6)--(\k+.1,2.4)--(\k-.1,2.4)--(\k-.1,2.6);
}
\foreach \k in {6.5} {
\fill[color=bead] (0,\k) circle (10pt);}
\foreach \k in {4.5,5.5,6.5} {
\fill[color=bead] (1,\k) circle (10pt);}
\foreach \k in {4.5,5.5,6.5} {
\fill[color=bead] (2,\k) circle (10pt);}
\foreach \k in {4.5,6.5} {
\fill[color=bead] (5,\k) circle (10pt);}
\foreach \k in {4.5,5.5,6.5} {
\fill[color=bead] (6,\k) circle (10pt);}
\foreach \k in {3.5,4.5,5.5,6.5} {
  \fill[color=bead] (7,\k) circle (10pt);}
\draw(-.5,7)--(2.5,7); \draw(4.5,7)--(7.5,7);
\node at (8.5,5) {,};
\foreach \k in {20,21,22,25,26,27} {
\fill [color=brown] (\k-.1,3)--(\k-.1,7) -- (\k+0.1,7)-- (\k+.1,3)--(\k-.1,3);
\fill[color=brown](\k-.1,7.1)--(\k+.1,7.1)--(\k+.1,7.3)--(\k-.1,7.3)--(\k-.1,7.1);
\fill[color=brown](\k-.1,2.9)--(\k+.1,2.9)--(\k+.1,2.7)--(\k-.1,2.7)--(\k-.1,2.9);
\fill[color=brown](\k-.1,2.6)--(\k+.1,2.6)--(\k+.1,2.4)--(\k-.1,2.4)--(\k-.1,2.6);
}
\foreach \k in {5.5,6.5} {
\fill[color=bead] (20,\k) circle (10pt);}
\foreach \k in {4.5,5.5,6.5} {
\fill[color=bead] (21,\k) circle (10pt);}
\foreach \k in {3.5,4.5,5.5,6.5} {
\fill[color=bead] (22,\k) circle (10pt);}
\foreach \k in {6.5} {
\fill[color=bead] (25,\k) circle (10pt);}
\foreach \k in {4.5,5.5,6.5} {
\fill[color=bead] (26,\k) circle (10pt);}
\foreach \k in {4.5,5.5,6.5} {
  \fill[color=bead] (27,\k) circle (10pt);}
\draw(19.5,7)--(22.5,7); \draw(24.5,7)--(27.5,7);
\node at (28.5,5) {.};
\foreach \a in {0,1,2,5,6,7,20,21,22,25,26,27} {
\foreach \k in {3.5,...,6.5}{
\node at (\a,\k)[color=bead]{$-$};};};
\end{tikzpicture}
\end{center}
So $\Phi_2(\la,s) = (((2^2,1^2),(3,1^5)),(7,9))$ and $\Phi_2(\mu,s) = (((3,1^2),(2^2,1^2)),(9,7))$ and $\Phi_2(\la,s) \not\ap_e \Phi_2(\mu,s)$.  
\end{ex}

Proposition~\ref{BlockPreserve} does immediately give us the following result.

\begin{corollary} \label{DU} 
Suppose that $\Ra$ is a $\ap_e$-equivalence class of $\A_e$. Then $\Phi_r(\Ra)$ is the disjoint union of $\ap_e$-equivalence classes in $\A_e^{(r)}$. 
\end{corollary}

In fact, it is not difficult to describe when the full converse to Proposition~\ref{BlockPreserve} holds. Recall the map $\eta:\A_e \rightarrow \La^{(e)} \times \Z^e$ that we introduced in Section~\ref{Blocks}. If $r=1$, the maps $\Psi_r$ and $\Phi_r$ are just the identity maps, so for the remainder of this section, we assume that $r \geq 2$. 

\begin{lemma}\label{EqSwap1}
Suppose that $\la,\mu \in \Lambda$ and that $\la \sim_e \mu$. Let $s \in \Z$. Then $\eta(\la,s) \xrightarrow{r}_1 \eta(\mu,s)$ if and only if $\Phi_r(\la,s) \xrightarrow{e}_1 \Phi_r(\mu,s)$.  
\end{lemma}

\begin{proof}
We have $\eta(\la,s) \xrightarrow{r}_1 \eta(\mu,s)$ if and only if there exist $0 \leq k_1,k_2 <e$ such that $(\mu,s)$ is formed from $(\la,s)$ by moving a bead on runner $k_1$ up by $r$ positions and moving a bead on runner $k_2$ down by $r$ positions. By Lemma~\ref{PsiMove}, this occurs if and only if $\Phi_r(\mu,s)$ is formed from $\Phi_r(\la,s)$ by moving a bead on runner $k_1$ of some component up by one position and moving a bead on runner $k_2$ of some component down by one position, that is, if and only if $\Phi_r(\la,s) \xrightarrow{e}_1 \Phi_r(\mu,s)$.
\end{proof}

\begin{lemma}\label{EqSwap2}
Suppose that $\la,\mu \in \Lambda$ and that $\la \sim_e \mu$. Let $s \in \Z$. Then $\eta(\la,s) \xrightarrow{r}_2 \eta(\mu,s)$ if and only if $\Phi_r(\la,s) \xrightarrow{e}_2 \Phi_r(\mu,s)$.  
\end{lemma}

Before reading the proof, it might be helpful to consider the following example.

\begin{ex}
Let $e=4$ and $r=3$ and suppose that $(\la,s),(\mu,s) \in \A_4$ have abacus diagrams respectively given by
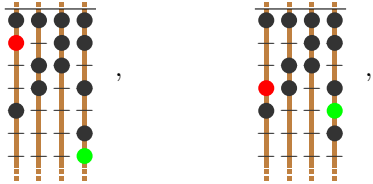
\begin{center}
\begin{tikzpicture}\tikzset{yscale=0.3,xscale=0.3}
\foreach \k in {0,1,2,3} {
\fill [color=brown] (\k-.1,0)--(\k-.1,7) -- (\k+0.1,7)-- (\k+.1,0)--(\k-.1,0);
\fill[color=brown](\k-.1,7.1)--(\k+.1,7.1)--(\k+.1,7.3)--(\k-.1,7.3)--(\k-.1,7.1);
\fill[color=brown](\k-.1,-.1)--(\k+.1,-.1)--(\k+.1,-.3)--(\k-.1,-.3)--(\k-.1,-.3);
\fill[color=brown](\k-.1,-.4)--(\k+.1,-.4)--(\k+.1,-.6)--(\k-.1,-.6)--(\k-.1,-.6);
}
\foreach \k in {2.5,6.5} {
\fill[color=bead] (0,\k) circle (10pt);}
\foreach \k in {3.5,4.5,6.5} {
\fill[color=bead] (1,\k) circle (10pt);}
\foreach \k in {4.5,5.5,6.5} {
\fill[color=bead] (2,\k) circle (10pt);}
\foreach \k in {6.5,5.5,3.5,1.5} {
\fill[color=bead] (3,\k) circle (10pt);}
\fill[color=red] (0,5.5) circle (10pt);
\fill[color=green] (3,0.5) circle (10pt);
\draw(-.5,7)--(3.5,7); 
\node at (4.5,4) {,};
\foreach \a in {0,1,2,3} {
\foreach \k in {1.5,...,4.5}{
\node at (\a,\k)[color=bead]{$-$};};};
\node at (1,5.5)[color=bead]{$-$};
\foreach \a in {0,1,2} { \node at (\a,0.5)[color=bead]{$-$};};
\end{tikzpicture}
\qquad \qquad 
\begin{tikzpicture}\tikzset{yscale=0.3,xscale=0.3}
\foreach \k in {0,1,2,3} {
\fill [color=brown] (\k-.1,0)--(\k-.1,7) -- (\k+0.1,7)-- (\k+.1,0)--(\k-.1,0);
\fill[color=brown](\k-.1,7.1)--(\k+.1,7.1)--(\k+.1,7.3)--(\k-.1,7.3)--(\k-.1,7.1);
\fill[color=brown](\k-.1,-.1)--(\k+.1,-.1)--(\k+.1,-.3)--(\k-.1,-.3)--(\k-.1,-.3);
\fill[color=brown](\k-.1,-.4)--(\k+.1,-.4)--(\k+.1,-.6)--(\k-.1,-.6)--(\k-.1,-.6);
}
\foreach \k in {2.5,6.5} {
\fill[color=bead] (0,\k) circle (10pt);}
\foreach \k in {3.5,4.5,6.5} {
\fill[color=bead] (1,\k) circle (10pt);}
\foreach \k in {4.5,5.5,6.5} {
\fill[color=bead] (2,\k) circle (10pt);}
\foreach \k in {6.5,5.5,3.5,1.5} {
\fill[color=bead] (3,\k) circle (10pt);}
\fill[color=red] (0,3.5) circle (10pt);
\fill[color=green] (3,2.5) circle (10pt);
\draw(-.5,7)--(3.5,7); 
\node at (4.5,4) {,};
\foreach \a in {0,1,2,3} {
\foreach \k in {0.5,1.5,4.5,5.5}{
\node at (\a,\k)[color=bead]{$-$};};};
\node at (1,2.5)[color=bead]{$-$}; \node at (2,2.5)[color=bead]{$-$}; \node at (2,3.5)[color=bead]{$-$};
\end{tikzpicture}
\end{center}
so that $\eta(\la,s) \xrightarrow{r}_2 \eta(\mu,s)$. Then $\Phi_r(\la,s)$ and $\Phi_r(\mu,s)$ are respectively given by

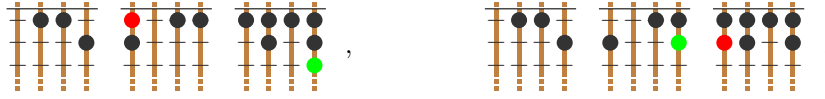
\begin{center}
\begin{tikzpicture} \tikzset{yscale=0.3,xscale=0.3}
\foreach \k in {0,1,2,3,5,6,7,8,10,11,12,13} {
\fill [color=brown] (\k-.1,2)--(\k-.1,5) -- (\k+0.1,5)-- (\k+.1,2)--(\k-.1,2);
\fill[color=brown](\k-.1,5.1)--(\k+.1,5.1)--(\k+.1,5.3)--(\k-.1,5.3)--(\k-.1,5.1);
\fill[color=brown](\k-.1,1.9)--(\k+.1,1.9)--(\k+.1,1.7)--(\k-.1,1.7)--(\k-.1,1.9);
\fill[color=brown](\k-.1,1.6)--(\k+.1,1.6)--(\k+.1,1.4)--(\k-.1,1.4)--(\k-.1,1.6);
}
\foreach \k in {} {
\fill[color=bead] (0,\k) circle (10pt);}
\foreach \k in {4.5} {
\fill[color=bead] (1,\k) circle (10pt);}
\foreach \k in {4.5} {
\fill[color=bead] (2,\k) circle (10pt);}
\foreach \k in {3.5} {
\fill[color=bead] (3,\k) circle (10pt);}
\foreach \k in {3.5} {
\fill[color=bead] (5,\k) circle (10pt);}
\foreach \k in {} {
  \fill[color=bead] (6,\k) circle (10pt);}
\foreach \k in {4.5} {
  \fill[color=bead] (7,\k) circle (10pt);}
\foreach \k in {4.5} {
  \fill[color=bead] (8,\k) circle (10pt);}
\foreach \k in {4.5} {
  \fill[color=bead] (10,\k) circle (10pt);}
\foreach \k in {3.5,4.5} {
  \fill[color=bead] (11,\k) circle (10pt);}
\foreach \k in {4.5} {
  \fill[color=bead] (12,\k) circle (10pt);}
\foreach \k in {3.5,4.5} {
  \fill[color=bead] (13,\k) circle (10pt);}
\fill[color=red] (5,4.5) circle (10pt);
\fill[color=green] (13,2.5) circle (10pt);
\draw(-.5,5)--(3.5,5); \draw(4.5,5)--(8.5,5); \draw(9.5,5)--(13.5,5); 
\node at (14.5,3){,}; 
\foreach \a in {0,3,6,7,8} {
\node at (\a,4.5)[color=bead]{$-$};};
\foreach \a in {0,1,2,6,7,8,10,12} {
\node at (\a,3.5)[color=bead]{$-$};};
\foreach \a in {0,1,2,3,5,6,7,8,10,11,12} {
\node at (\a,2.5)[color=bead]{$-$};};
\end{tikzpicture}
\qquad \qquad 
\begin{tikzpicture} \tikzset{yscale=0.3,xscale=0.3}
\foreach \k in {0,1,2,3,5,6,7,8,10,11,12,13} {
\fill [color=brown] (\k-.1,2)--(\k-.1,5) -- (\k+0.1,5)-- (\k+.1,2)--(\k-.1,2);
\fill[color=brown](\k-.1,5.1)--(\k+.1,5.1)--(\k+.1,5.3)--(\k-.1,5.3)--(\k-.1,5.1);
\fill[color=brown](\k-.1,1.9)--(\k+.1,1.9)--(\k+.1,1.7)--(\k-.1,1.7)--(\k-.1,1.9);
\fill[color=brown](\k-.1,1.6)--(\k+.1,1.6)--(\k+.1,1.4)--(\k-.1,1.4)--(\k-.1,1.6);
}
\foreach \k in {} {
\fill[color=bead] (0,\k) circle (10pt);}
\foreach \k in {4.5} {
\fill[color=bead] (1,\k) circle (10pt);}
\foreach \k in {4.5} {
\fill[color=bead] (2,\k) circle (10pt);}
\foreach \k in {3.5} {
\fill[color=bead] (3,\k) circle (10pt);}
\foreach \k in {3.5} {
\fill[color=bead] (5,\k) circle (10pt);}
\foreach \k in {} {
  \fill[color=bead] (6,\k) circle (10pt);}
\foreach \k in {4.5} {
  \fill[color=bead] (7,\k) circle (10pt);}
\foreach \k in {4.5} {
  \fill[color=bead] (8,\k) circle (10pt);}
\foreach \k in {4.5} {
  \fill[color=bead] (10,\k) circle (10pt);}
\foreach \k in {3.5,4.5} {
  \fill[color=bead] (11,\k) circle (10pt);}
\foreach \k in {4.5} {
  \fill[color=bead] (12,\k) circle (10pt);}
\foreach \k in {4.5,3.5} {
  \fill[color=bead] (13,\k) circle (10pt);}
\fill[color=red] (10,3.5) circle (10pt);
\fill[color=green] (8,3.5) circle (10pt);
\draw(-.5,5)--(3.5,5); \draw(4.5,5)--(8.5,5); \draw(9.5,5)--(13.5,5); 
\node at (14.5,3){,}; 
\foreach \a in {0,3,5,6,7,8} {
\node at (\a,4.5)[color=bead]{$-$};};
\foreach \a in {0,1,2,6,7,12} {
\node at (\a,3.5)[color=bead]{$-$};};
\foreach \a in {0,1,2,3,5,6,7,8,10,11,12,13} {
\node at (\a,2.5)[color=bead]{$-$};};
\end{tikzpicture}
\end{center}
and we can see that $\Phi_r(\la,s) \xrightarrow{e}_2 \Phi_r(\mu,s)$. The condition that $\eta(\la,s) \xrightarrow{r}_2 \eta(\mu,s)$ ensures that in the second set of pictures if the red bead jumps from component $x$ to component $y$ then green bead jumps from component $y$ to component $x$.  
\end{ex}

\begin{proof}[Proof of Lemma~\ref{EqSwap2}]
Suppose $\eta(\la,s) \xrightarrow{r}_2 \eta(\mu,s)$. Then there exist $0 \leq k_1,k_2 < e$, $b_1,b_2 \in \Z$ and $h>0$ such that $b_1 \equiv b_2 \mod r$ and $(\mu,s)$ is formed from $(\la,s)$ by moving the bead at $z_1=b_1 e+ k_1$ to $(b_1+h)e+k_1$ and the bead at $(b_2+h)e+k_2$ to $z_2=b_2e+k_2$. 
Write $b_1=m_1r+t$ and $b_2=m_2 r+t$ where $0 \leq t <r$.  Suppose that $h=h^\ast r +h'$ where $0 \leq h'<r$.
So $\psi_{r-t}(m_1e+k_1)=z_1$ and $\psi_{r-t}(m_2 e+k_2)=z_2$.
Suppose that $h'+t<r$ so that 
\begin{align*}
z_1 + he & = (b_1+h)e+k_1 \\ & = (m_1r+t+h^\ast r + h')e + k_1 &  \\
& = ((m_1+h^\ast)r + t+h')e+k_1 \\
& = \psi_{r-t-h'}((m_1+h^\ast)e+k_1)
\end{align*}
and this combined with other computations gives us 
\begin{align*}
z_1+he & = \psi_{r-t-h'}((m_1+h^\ast)e+k_1),& z_1 & = \psi_{r-t}(m_1e+k_1),\\
z_2 + he & = \psi_{r-t-h'}((m_2+h^\ast)e+k_2), & z_2 &= \psi_{r-t}(m_2e+k_2).
\end{align*}

Then $\Phi_r(\mu,s)$ is formed from $\Phi_r(\la,s)$ by moving the bead of $\beta$-number $m_1e+k_1$ on component $r-t$ to position $(m_1+h^\ast)e+k_1$ of component $r-t-h'$ and moving the bead at $(m_2+h^\ast)e+k_2$ on component $r-t-h'$ to position $m_2 e+k_2$ on component $r-t$. Or, to put it another way, on abacus $r-t$ we replace a bead at $m_1e+k_1$ with a bead at $m_2e+k_2$ and on abacus $r-t-h'$ we replace a bead at $(m_2+h^\ast)e+k_2$ with a bead at $(m_1+h^\ast)e + k_1$. Thus $\Phi_r(\la,s) \xrightarrow{e}_2 \Phi_r(\mu,s)$. The argument is almost identical if $h'+t \geq r$. 

If $\Phi_r(\la,s) \xrightarrow{e}_2 \Phi_r(\mu,s)$ then reversing the argument above shows that $\eta(\la,s) \xrightarrow{r}_2 \eta(\la,s)$. 
\end{proof}

\begin{theorem}\label{EqEq}
Suppose that $\la,\mu \in \Lambda$ and that $\la \sim_e \mu$. Let $s \in \Z$. 
Then $\eta(\la,s) \ap_r \eta(\mu,s)$ if and only if $\Phi_r(\la,s) \ap_e \Phi_r(\mu,s)$.
\end{theorem}

\begin{proof}
Using Proposition~\ref{GenBlock}, the result follows immediately from Lemma~\ref{EqSwap1} and Lemma~\ref{EqSwap2}. 
\end{proof}

Proposition~\ref{BlockPreserve} shows that the map $\Psi_r$ sends a block into a single block while its inverse $\Phi_r$ splits a block into a disjoint union of blocks, with this disjoint union determined by Theorem~\ref{EqEq}. 

\begin{ex}
Let $r=2$ and $e=3$. Take $s=16$ and consider $(\la,s) \in \A_3$ given by 
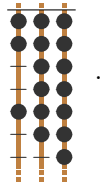
\begin{center}
\begin{tikzpicture}\tikzset{yscale=0.3,xscale=0.3}
\foreach \k in {0,1,2} {
\fill [color=brown] (\k-.1,0)--(\k-.1,7) -- (\k+0.1,7)-- (\k+.1,0)--(\k-.1,0);
\fill[color=brown](\k-.1,7.1)--(\k+.1,7.1)--(\k+.1,7.3)--(\k-.1,7.3)--(\k-.1,7.1);
\fill[color=brown](\k-.1,-.1)--(\k+.1,-.1)--(\k+.1,-.3)--(\k-.1,-.3)--(\k-.1,-.3);
\fill[color=brown](\k-.1,-.4)--(\k+.1,-.4)--(\k+.1,-.6)--(\k-.1,-.6)--(\k-.1,-.4);
}
\foreach \k in {2.5,5.5,6.5} {
\fill[color=bead] (0,\k) circle (10pt);}
\foreach \k in {1.5,2.5,3.5,4.5,5.5,6.5} {
\fill[color=bead] (1,\k) circle (10pt);}
\foreach \k in {0.5,1.5,2.5,3.5,4.5,5.5,6.5} {
\fill[color=bead] (2,\k) circle (10pt);}
\draw(-.5,7) -- (2.5,7);
\node at (3.5,4) {.};
\foreach \a in {0,1,2} {
\foreach \k in {0.5,...,6.5}{
\node at (\a,\k)[color=bead]{$-$};};};
\end{tikzpicture}
\end{center}
The $\sim_3$-equivalence class of $\la$ is in bijection with $\La^{(3)}_2$ via the map $\eta$:
$$\la \sim_3 \mu \iff (\mu,s) \ap_3 (\la,s) \iff \eta(\mu,s) = (\brho,(3,6,7))  \text{ for some } \brho \in \La^{(3)}_2.$$ We consider the $\sim_{2,(3,6,7)}$-equivalence classes, or (equivalently) the $\sim_{2,(1,0,1)}$-equivalence classes of $\La^{(3)}_2$; since we are looking at multipartitions of 2, it is easy to write down the two classes using their residue diagrams:   
\begin{align*}
\Bigg\{ \; & \young(10)\;\emp\;\emp,& &
\young(1,0)\;\emp\;\emp, & &
\emp\;\young(01)\;\emp, & &
\emp\;\young(0,1)\;\emp, \\
&\emp\;\emp\;\young(10)\,, & &
\emp\;\emp\;\young(1,0)\,, & &
\young(1)\;\young(0)\;\emp, & &
\emp\;\young(0)\;\young(1) \; \Bigg\},
\end{align*}
\[\Bigg\{\; \young(1)\;\emp\;\young(1)\; \Bigg\}.\]

Below we give the abacus configuration of each $(\mu,s) \ap_3 (\la,s)$ followed by its image under $\Phi_2$. 

\bigskip
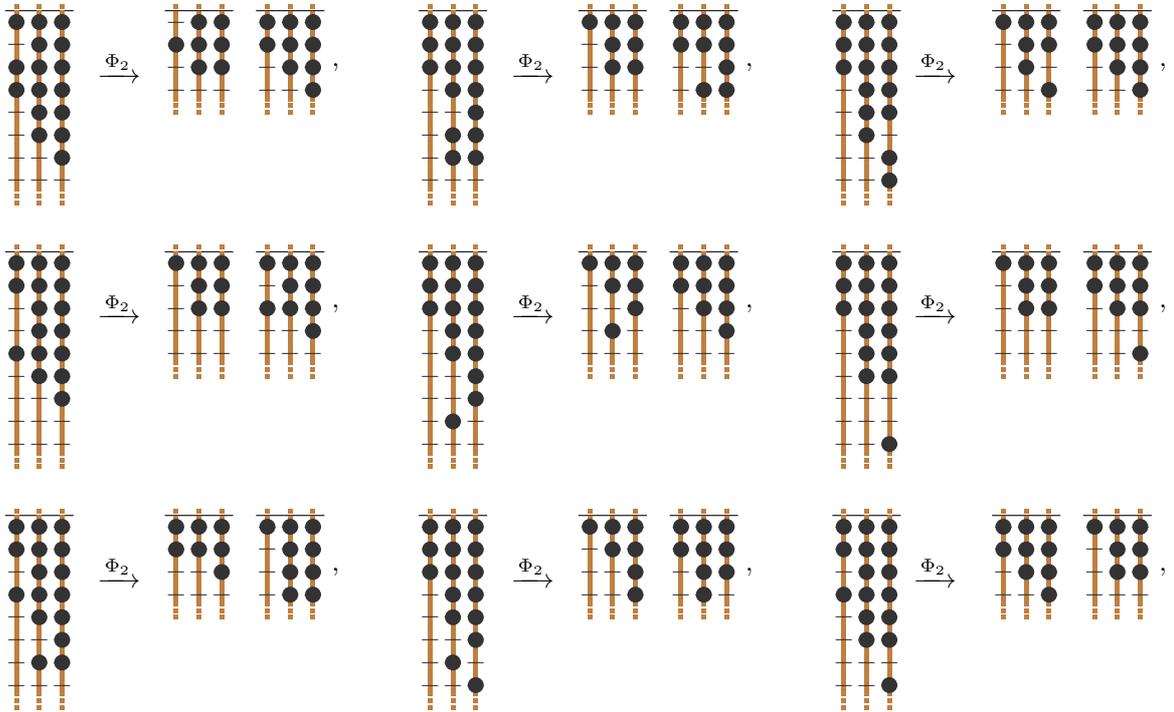
\begin{tikzpicture}\tikzset{yscale=0.3,xscale=0.3}
\draw(2.5,16)--(5.5,16);
\foreach \k in {3,4,5} {
\fill [color=brown] (\k-.1,8)--(\k-.1,16) -- (\k+0.1,16)-- (\k+.1,8)--(\k-.1,8);
\fill[color=brown](\k-.1,16.1)--(\k+.1,16.1)--(\k+.1,16.3)--(\k-.1,16.3)--(\k-.1,16.1);
\fill[color=brown](\k-.1,7.9)--(\k+.1,7.9)--(\k+.1,7.7)--(\k-.1,7.7)--(\k-.1,7.9);
\fill[color=brown](\k-.1,7.6)--(\k+.1,7.6)--(\k+.1,7.4)--(\k-.1,7.4)--(\k-.1,7.6);
}
\foreach \k in {12.5,13.5,15.5} {
\fill[color=bead] (3,\k) circle (10pt);}
\foreach \k in {10.5,11.5,12.5,13.5,14.5,15.5} {
\fill[color=bead] (4,\k) circle (10pt);}
\foreach \k in {9.5,10.5,11.5,12.5,13.5,14.5,15.5} {
\fill[color=bead] (5,\k) circle (10pt);}
\node at (7.5,13.5){$\xrightarrow{\Phi_2}$}; 
\draw(9.5,16)--(12.5,16); \draw(13.5,16)--(16.5,16); 
\foreach \k in {10,11,12,14,15,16} {
\fill [color=brown] (\k-.1,12)--(\k-.1,16) -- (\k+0.1,16)-- (\k+.1,12)--(\k-.1,12);
\fill[color=brown](\k-.1,16.1)--(\k+.1,16.1)--(\k+.1,16.3)--(\k-.1,16.3)--(\k-.1,16.1);
\fill[color=brown](\k-.1,11.9)--(\k+.1,11.9)--(\k+.1,11.7)--(\k-.1,11.7)--(\k-.1,11.9);
\fill[color=brown](\k-.1,11.6)--(\k+.1,11.6)--(\k+.1,11.4)--(\k-.1,11.4)--(\k-.1,11.6);
}
\foreach \k in {14.5} {
\fill[color=bead] (10,\k) circle (10pt);}
\foreach \k in {13.5,14.5,15.5} {
\fill[color=bead] (11,\k) circle (10pt);}
\foreach \k in {13.5,14.5,15.5} {
\fill[color=bead] (12,\k) circle (10pt);}
\foreach \k in {14.5,15.5} {
\fill[color=bead] (14,\k) circle (10pt);}
\foreach \k in {13.5,14.5,15.5} {
\fill[color=bead] (15,\k) circle (10pt);}
\foreach \k in {12.5,13.5,14.5,15.5} {
  \fill[color=bead] (16,\k) circle (10pt);}
\foreach \a in {3,4,5} {
\foreach \k in {8.5,...,15.5}{
\node at (\a,\k)[color=bead]{$-$};};};
\foreach \a in {10,11,12,14,15,16} {
\foreach \k in {12.5,...,15.5}{
\node at (\a,\k)[color=bead]{$-$};};};
\node at (17,13.5){,};
\end{tikzpicture}
\qquad
\begin{tikzpicture}\tikzset{yscale=0.3,xscale=0.3}
\draw(2.5,16)--(5.5,16);
\foreach \k in {3,4,5} {
\fill [color=brown] (\k-.1,8)--(\k-.1,16) -- (\k+0.1,16)-- (\k+.1,8)--(\k-.1,8);
\fill[color=brown](\k-.1,16.1)--(\k+.1,16.1)--(\k+.1,16.3)--(\k-.1,16.3)--(\k-.1,16.1);
\fill[color=brown](\k-.1,7.9)--(\k+.1,7.9)--(\k+.1,7.7)--(\k-.1,7.7)--(\k-.1,7.9);
\fill[color=brown](\k-.1,7.6)--(\k+.1,7.6)--(\k+.1,7.4)--(\k-.1,7.4)--(\k-.1,7.6);
}
\foreach \k in {13.5,14.5,15.5} {
\fill[color=bead] (3,\k) circle (10pt);}
\foreach \k in {9.5,10.5,12.5,13.5,14.5,15.5} {
\fill[color=bead] (4,\k) circle (10pt);}
\foreach \k in {9.5,10.5,11.5,12.5,13.5,14.5,15.5} {
\fill[color=bead] (5,\k) circle (10pt);}
\node at (7.5,13.5){$\xrightarrow{\Phi_2}$}; 
\draw(9.5,16)--(12.5,16); \draw(13.5,16)--(16.5,16); 
\foreach \k in {10,11,12,14,15,16} {
\fill [color=brown] (\k-.1,12)--(\k-.1,16) -- (\k+0.1,16)-- (\k+.1,12)--(\k-.1,12);
\fill[color=brown](\k-.1,16.1)--(\k+.1,16.1)--(\k+.1,16.3)--(\k-.1,16.3)--(\k-.1,16.1);
\fill[color=brown](\k-.1,11.9)--(\k+.1,11.9)--(\k+.1,11.7)--(\k-.1,11.7)--(\k-.1,11.9);
\fill[color=brown](\k-.1,11.6)--(\k+.1,11.6)--(\k+.1,11.4)--(\k-.1,11.4)--(\k-.1,11.6);
}
\foreach \k in {15.5} {
\fill[color=bead] (10,\k) circle (10pt);}
\foreach \k in {13.5,14.5,15.5} {
\fill[color=bead] (11,\k) circle (10pt);}
\foreach \k in {13.5,14.5,15.5} {
\fill[color=bead] (12,\k) circle (10pt);}
\foreach \k in {14.5,15.5} {
\fill[color=bead] (14,\k) circle (10pt);}
\foreach \k in {12.5,14.5,15.5} {
\fill[color=bead] (15,\k) circle (10pt);}
\foreach \k in {12.5,13.5,14.5,15.5} {
  \fill[color=bead] (16,\k) circle (10pt);}
\foreach \a in {3,4,5} {
\foreach \k in {8.5,...,15.5}{
\node at (\a,\k)[color=bead]{$-$};};};
\foreach \a in {10,11,12,14,15,16} {
\foreach \k in {12.5,...,15.5}{
\node at (\a,\k)[color=bead]{$-$};};};
\node at (17,13.5){,};
\end{tikzpicture}
\qquad
\begin{tikzpicture}\tikzset{yscale=0.3,xscale=0.3}
\draw(2.5,16)--(5.5,16);
\foreach \k in {3,4,5} {
\fill [color=brown] (\k-.1,8)--(\k-.1,16) -- (\k+0.1,16)-- (\k+.1,8)--(\k-.1,8);
\fill[color=brown](\k-.1,16.1)--(\k+.1,16.1)--(\k+.1,16.3)--(\k-.1,16.3)--(\k-.1,16.1);
\fill[color=brown](\k-.1,7.9)--(\k+.1,7.9)--(\k+.1,7.7)--(\k-.1,7.7)--(\k-.1,7.9);
\fill[color=brown](\k-.1,7.6)--(\k+.1,7.6)--(\k+.1,7.4)--(\k-.1,7.4)--(\k-.1,7.6);
}
\foreach \k in {13.5,14.5,15.5} {
\fill[color=bead] (3,\k) circle (10pt);}
\foreach \k in {10.5,11.5,12.5,13.5,14.5,15.5} {
\fill[color=bead] (4,\k) circle (10pt);}
\foreach \k in {8.5,9.5,11.5,12.5,13.5,14.5,15.5} {
\fill[color=bead] (5,\k) circle (10pt);}
\node at (7,13.5){$\xrightarrow{\Phi_2}$}; 
\draw(9.5,16)--(12.5,16); \draw(13.5,16)--(16.5,16); 
\foreach \k in {10,11,12,14,15,16} {
\fill [color=brown] (\k-.1,12)--(\k-.1,16) -- (\k+0.1,16)-- (\k+.1,12)--(\k-.1,12);
\fill[color=brown](\k-.1,16.1)--(\k+.1,16.1)--(\k+.1,16.3)--(\k-.1,16.3)--(\k-.1,16.1);
\fill[color=brown](\k-.1,11.9)--(\k+.1,11.9)--(\k+.1,11.7)--(\k-.1,11.7)--(\k-.1,11.9);
\fill[color=brown](\k-.1,11.6)--(\k+.1,11.6)--(\k+.1,11.4)--(\k-.1,11.4)--(\k-.1,11.6);
}
\foreach \k in {15.5} {
\fill[color=bead] (10,\k) circle (10pt);}
\foreach \k in {13.5,14.5,15.5} {
\fill[color=bead] (11,\k) circle (10pt);}
\foreach \k in {12.5,14.5,15.5} {
\fill[color=bead] (12,\k) circle (10pt);}
\foreach \k in {14.5,15.5} {
\fill[color=bead] (14,\k) circle (10pt);}
\foreach \k in {13.5,14.5,15.5} {
\fill[color=bead] (15,\k) circle (10pt);}
\foreach \k in {12.5,13.5,14.5,15.5} {
  \fill[color=bead] (16,\k) circle (10pt);}
\foreach \a in {3,4,5} {
\foreach \k in {8.5,...,15.5}{
\node at (\a,\k)[color=bead]{$-$};};};
\foreach \a in {10,11,12,14,15,16} {
\foreach \k in {12.5,...,15.5}{
\node at (\a,\k)[color=bead]{$-$};};};
\node at (17,13.5){,};
\end{tikzpicture}

\bigskip
\begin{tikzpicture}\tikzset{yscale=0.3,xscale=0.3}
\draw(2.5,16)--(5.5,16);
\foreach \k in {3,4,5} {
\fill [color=brown] (\k-.1,7)--(\k-.1,16) -- (\k+0.1,16)-- (\k+.1,7)--(\k-.1,7);
\fill[color=brown](\k-.1,16.1)--(\k+.1,16.1)--(\k+.1,16.3)--(\k-.1,16.3)--(\k-.1,16.1);
\fill[color=brown](\k-.1,6.9)--(\k+.1,6.9)--(\k+.1,6.7)--(\k-.1,6.7)--(\k-.1,6.9);
\fill[color=brown](\k-.1,6.6)--(\k+.1,6.6)--(\k+.1,6.4)--(\k-.1,6.4)--(\k-.1,6.6);
}
\foreach \k in {11.5,14.5,15.5} {
\fill[color=bead] (3,\k) circle (10pt);}
\foreach \k in {10.5,11.5,12.5,13.5,14.5,15.5} {
\fill[color=bead] (4,\k) circle (10pt);}
\foreach \k in {9.5,10.5,11.5,12.5,13.5,14.5,15.5} {
\fill[color=bead] (5,\k) circle (10pt);}
\node at (7.5,13.5){$\xrightarrow{\Phi_2}$}; 
\draw(9.5,16)--(12.5,16); \draw(13.5,16)--(16.5,16); 
\foreach \k in {10,11,12,14,15,16} {
\fill [color=brown] (\k-.1,11)--(\k-.1,16) -- (\k+0.1,16)-- (\k+.1,11)--(\k-.1,11);
\fill[color=brown](\k-.1,16.1)--(\k+.1,16.1)--(\k+.1,16.3)--(\k-.1,16.3)--(\k-.1,16.1);
\fill[color=brown](\k-.1,10.9)--(\k+.1,10.9)--(\k+.1,10.7)--(\k-.1,10.7)--(\k-.1,10.9);
\fill[color=brown](\k-.1,10.6)--(\k+.1,10.6)--(\k+.1,10.4)--(\k-.1,10.4)--(\k-.1,10.6);
}
\foreach \k in {15.5} {
\fill[color=bead] (10,\k) circle (10pt);}
\foreach \k in {13.5,14.5,15.5} {
\fill[color=bead] (11,\k) circle (10pt);}
\foreach \k in {13.5,14.5,15.5} {
\fill[color=bead] (12,\k) circle (10pt);}
\foreach \k in {13.5,15.5} {
\fill[color=bead] (14,\k) circle (10pt);}
\foreach \k in {13.5,14.5,15.5} {
\fill[color=bead] (15,\k) circle (10pt);}
\foreach \k in {12.5,13.5,14.5,15.5} {
  \fill[color=bead] (16,\k) circle (10pt);}
\foreach \a in {3,4,5} {
\foreach \k in {7.5,...,15.5}{
\node at (\a,\k)[color=bead]{$-$};};};
\foreach \a in {10,11,12,14,15,16} {
\foreach \k in {11.5,12.5,...,15.5}{
\node at (\a,\k)[color=bead]{$-$};};};
\node at (17,13.5){,};
\end{tikzpicture}
\qquad
\begin{tikzpicture}\tikzset{yscale=0.3,xscale=0.3}
\draw(2.5,16)--(5.5,16);
\foreach \k in {3,4,5} {
\fill [color=brown] (\k-.1,7)--(\k-.1,16) -- (\k+0.1,16)-- (\k+.1,7)--(\k-.1,7);
\fill[color=brown](\k-.1,16.1)--(\k+.1,16.1)--(\k+.1,16.3)--(\k-.1,16.3)--(\k-.1,16.1);
\fill[color=brown](\k-.1,6.9)--(\k+.1,6.9)--(\k+.1,6.7)--(\k-.1,6.7)--(\k-.1,6.9);
\fill[color=brown](\k-.1,6.6)--(\k+.1,6.6)--(\k+.1,6.4)--(\k-.1,6.4)--(\k-.1,6.6);
}
\foreach \k in {13.5,14.5,15.5} {
\fill[color=bead] (3,\k) circle (10pt);}
\foreach \k in {8.5,11.5,12.5,13.5,14.5,15.5} {
\fill[color=bead] (4,\k) circle (10pt);}
\foreach \k in {9.5,10.5,11.5,12.5,13.5,14.5,15.5} {
\fill[color=bead] (5,\k) circle (10pt);}
\node at (7.5,13.5){$\xrightarrow{\Phi_2}$}; 
\draw(9.5,16)--(12.5,16); \draw(13.5,16)--(16.5,16); 
\foreach \k in {10,11,12,14,15,16} {
\fill [color=brown] (\k-.1,11)--(\k-.1,16) -- (\k+0.1,16)-- (\k+.1,11)--(\k-.1,11);
\fill[color=brown](\k-.1,16.1)--(\k+.1,16.1)--(\k+.1,16.3)--(\k-.1,16.3)--(\k-.1,16.1);
\fill[color=brown](\k-.1,10.9)--(\k+.1,10.9)--(\k+.1,10.7)--(\k-.1,10.7)--(\k-.1,10.9);
\fill[color=brown](\k-.1,10.6)--(\k+.1,10.6)--(\k+.1,10.4)--(\k-.1,10.4)--(\k-.1,10.6);
}
\foreach \k in {15.5} {
\fill[color=bead] (10,\k) circle (10pt);}
\foreach \k in {12.5,14.5,15.5} {
\fill[color=bead] (11,\k) circle (10pt);}
\foreach \k in {13.5,14.5,15.5} {
\fill[color=bead] (12,\k) circle (10pt);}
\foreach \k in {14.5,15.5} {
\fill[color=bead] (14,\k) circle (10pt);}
\foreach \k in {13.5,14.5,15.5} {
\fill[color=bead] (15,\k) circle (10pt);}
\foreach \k in {12.5,13.5,14.5,15.5} {
  \fill[color=bead] (16,\k) circle (10pt);}
\foreach \a in {3,4,5} {
\foreach \k in {7.5,...,15.5}{
\node at (\a,\k)[color=bead]{$-$};};};
\foreach \a in {10,11,12,14,15,16} {
\foreach \k in {11.5,...,15.5}{
\node at (\a,\k)[color=bead]{$-$};};};
\node at (17,13.5){,};
\end{tikzpicture}
\qquad
\begin{tikzpicture}\tikzset{yscale=0.3,xscale=0.3}
\draw(2.5,16)--(5.5,16);
\foreach \k in {3,4,5} {
\fill [color=brown] (\k-.1,7)--(\k-.1,16) -- (\k+0.1,16)-- (\k+.1,7)--(\k-.1,7);
\fill[color=brown](\k-.1,16.1)--(\k+.1,16.1)--(\k+.1,16.3)--(\k-.1,16.3)--(\k-.1,16.1);
\fill[color=brown](\k-.1,6.9)--(\k+.1,6.9)--(\k+.1,6.7)--(\k-.1,6.7)--(\k-.1,6.9);
\fill[color=brown](\k-.1,6.6)--(\k+.1,6.6)--(\k+.1,6.4)--(\k-.1,6.4)--(\k-.1,6.6);
}
\foreach \k in {13.5,14.5,15.5} {
\fill[color=bead] (3,\k) circle (10pt);}
\foreach \k in {10.5,11.5,12.5,13.5,14.5,15.5} {
\fill[color=bead] (4,\k) circle (10pt);}
\foreach \k in {7.5,10.5,11.5,12.5,13.5,14.5,15.5} {
\fill[color=bead] (5,\k) circle (10pt);}
\node at (7,13.5){$\xrightarrow{\Phi_2}$}; 
\draw(9.5,16)--(12.5,16); \draw(13.5,16)--(16.5,16); 
\foreach \k in {10,11,12,14,15,16} {
\fill [color=brown] (\k-.1,11)--(\k-.1,16) -- (\k+0.1,16)-- (\k+.1,11)--(\k-.1,11);
\fill[color=brown](\k-.1,16.1)--(\k+.1,16.1)--(\k+.1,16.3)--(\k-.1,16.3)--(\k-.1,16.1);
\fill[color=brown](\k-.1,10.9)--(\k+.1,10.9)--(\k+.1,10.7)--(\k-.1,10.7)--(\k-.1,10.9);
\fill[color=brown](\k-.1,10.6)--(\k+.1,10.6)--(\k+.1,10.4)--(\k-.1,10.4)--(\k-.1,10.6);
}
\foreach \k in {15.5} {
\fill[color=bead] (10,\k) circle (10pt);}
\foreach \k in {13.5,14.5,15.5} {
\fill[color=bead] (11,\k) circle (10pt);}
\foreach \k in {13.5,14.5,15.5} {
\fill[color=bead] (12,\k) circle (10pt);}
\foreach \k in {14.5,15.5} {
\fill[color=bead] (14,\k) circle (10pt);}
\foreach \k in {13.5,14.5,15.5} {
\fill[color=bead] (15,\k) circle (10pt);}
\foreach \k in {11.5,13.5,14.5,15.5} {
  \fill[color=bead] (16,\k) circle (10pt);}
\foreach \a in {3,4,5} {
\foreach \k in {7.5,...,15.5}{
\node at (\a,\k)[color=bead]{$-$};};};
\foreach \a in {10,11,12,14,15,16} {
\foreach \k in {11.5,...,15.5}{
\node at (\a,\k)[color=bead]{$-$};};};
\node at (17,13.5){,};
\end{tikzpicture}

\bigskip
\begin{tikzpicture}\tikzset{yscale=0.3,xscale=0.3}
\draw(2.5,16)--(5.5,16);
\foreach \k in {3,4,5} {
\fill [color=brown] (\k-.1,8)--(\k-.1,16) -- (\k+0.1,16)-- (\k+.1,8)--(\k-.1,8);
\fill[color=brown](\k-.1,16.1)--(\k+.1,16.1)--(\k+.1,16.3)--(\k-.1,16.3)--(\k-.1,16.1);
\fill[color=brown](\k-.1,7.9)--(\k+.1,7.9)--(\k+.1,7.7)--(\k-.1,7.7)--(\k-.1,7.9);
\fill[color=brown](\k-.1,7.6)--(\k+.1,7.6)--(\k+.1,7.4)--(\k-.1,7.4)--(\k-.1,7.6);
}
\foreach \k in {12.5,14.5,15.5} {
\fill[color=bead] (3,\k) circle (10pt);}
\foreach \k in {9.5,11.5,12.5,13.5,14.5,15.5} {
\fill[color=bead] (4,\k) circle (10pt);}
\foreach \k in {9.5,10.5,11.5,12.5,13.5,14.5,15.5} {
\fill[color=bead] (5,\k) circle (10pt);}
\node at (7.5,13.5){$\xrightarrow{\Phi_2}$}; 
\draw(9.5,16)--(12.5,16); \draw(13.5,16)--(16.5,16); 
\foreach \k in {10,11,12,14,15,16} {
\fill [color=brown] (\k-.1,12)--(\k-.1,16) -- (\k+0.1,16)-- (\k+.1,12)--(\k-.1,12);
\fill[color=brown](\k-.1,16.1)--(\k+.1,16.1)--(\k+.1,16.3)--(\k-.1,16.3)--(\k-.1,16.1);
\fill[color=brown](\k-.1,11.9)--(\k+.1,11.9)--(\k+.1,11.7)--(\k-.1,11.7)--(\k-.1,11.9);
\fill[color=brown](\k-.1,11.6)--(\k+.1,11.6)--(\k+.1,11.4)--(\k-.1,11.4)--(\k-.1,11.6);
}
\foreach \k in {14.5,15.5} {
\fill[color=bead] (10,\k) circle (10pt);}
\foreach \k in {14.5,15.5} {
\fill[color=bead] (11,\k) circle (10pt);}
\foreach \k in {13.5,14.5,15.5} {
\fill[color=bead] (12,\k) circle (10pt);}
\foreach \k in {15.5} {
\fill[color=bead] (14,\k) circle (10pt);}
\foreach \k in {12.5,13.5,14.5,15.5} {
\fill[color=bead] (15,\k) circle (10pt);}
\foreach \k in {12.5,13.5,14.5,15.5} {
  \fill[color=bead] (16,\k) circle (10pt);}
\foreach \a in {3,4,5} {
\foreach \k in {8.5,...,15.5}{
\node at (\a,\k)[color=bead]{$-$};};};
\foreach \a in {10,11,12,14,15,16} {
\foreach \k in {12.5,...,15.5}{
\node at (\a,\k)[color=bead]{$-$};};};
\node at (17,13.5){,};
\end{tikzpicture}
\qquad
\begin{tikzpicture}\tikzset{yscale=0.3,xscale=0.3}
\draw(2.5,16)--(5.5,16);
\foreach \k in {3,4,5} {
\fill [color=brown] (\k-.1,8)--(\k-.1,16) -- (\k+0.1,16)-- (\k+.1,8)--(\k-.1,8);
\fill[color=brown](\k-.1,16.1)--(\k+.1,16.1)--(\k+.1,16.3)--(\k-.1,16.3)--(\k-.1,16.1);
\fill[color=brown](\k-.1,7.9)--(\k+.1,7.9)--(\k+.1,7.7)--(\k-.1,7.7)--(\k-.1,7.9);
\fill[color=brown](\k-.1,7.6)--(\k+.1,7.6)--(\k+.1,7.4)--(\k-.1,7.4)--(\k-.1,7.6);
}
\foreach \k in {13.5,14.5,15.5} {
\fill[color=bead] (3,\k) circle (10pt);}
\foreach \k in {9.5,11.5,12.5,13.5,14.5,15.5} {
\fill[color=bead] (4,\k) circle (10pt);}
\foreach \k in {8.5,10.5,11.5,12.5,13.5,14.5,15.5} {
\fill[color=bead] (5,\k) circle (10pt);}
\node at (7.5,13.5){$\xrightarrow{\Phi_2}$}; 
\draw(9.5,16)--(12.5,16); \draw(13.5,16)--(16.5,16); 
\foreach \k in {10,11,12,14,15,16} {
\fill [color=brown] (\k-.1,12)--(\k-.1,16) -- (\k+0.1,16)-- (\k+.1,12)--(\k-.1,12);
\fill[color=brown](\k-.1,16.1)--(\k+.1,16.1)--(\k+.1,16.3)--(\k-.1,16.3)--(\k-.1,16.1);
\fill[color=brown](\k-.1,11.9)--(\k+.1,11.9)--(\k+.1,11.7)--(\k-.1,11.7)--(\k-.1,11.9);
\fill[color=brown](\k-.1,11.6)--(\k+.1,11.6)--(\k+.1,11.4)--(\k-.1,11.4)--(\k-.1,11.6);
}
\foreach \k in {15.5} {
\fill[color=bead] (10,\k) circle (10pt);}
\foreach \k in {14.5,15.5} {
\fill[color=bead] (11,\k) circle (10pt);}
\foreach \k in {12.5,13.5,14.5,15.5} {
\fill[color=bead] (12,\k) circle (10pt);}
\foreach \k in {14.5,15.5} {
\fill[color=bead] (14,\k) circle (10pt);}
\foreach \k in {12.5,13.5,14.5,15.5} {
\fill[color=bead] (15,\k) circle (10pt);}
\foreach \k in {13.5,14.5,15.5} {
  \fill[color=bead] (16,\k) circle (10pt);}
\foreach \a in {3,4,5} {
\foreach \k in {8.5,...,15.5}{
\node at (\a,\k)[color=bead]{$-$};};};
\foreach \a in {10,11,12,14,15,16} {
\foreach \k in {12.5,...,15.5}{
\node at (\a,\k)[color=bead]{$-$};};};
\node at (17,13.5){,};
\end{tikzpicture}
\qquad
\begin{tikzpicture}\tikzset{yscale=0.3,xscale=0.3}
\draw(2.5,16)--(5.5,16);
\foreach \k in {3,4,5} {
\fill [color=brown] (\k-.1,8)--(\k-.1,16) -- (\k+0.1,16)-- (\k+.1,8)--(\k-.1,8);
\fill[color=brown](\k-.1,16.1)--(\k+.1,16.1)--(\k+.1,16.3)--(\k-.1,16.3)--(\k-.1,16.1);
\fill[color=brown](\k-.1,7.9)--(\k+.1,7.9)--(\k+.1,7.7)--(\k-.1,7.7)--(\k-.1,7.9);
\fill[color=brown](\k-.1,7.6)--(\k+.1,7.6)--(\k+.1,7.4)--(\k-.1,7.4)--(\k-.1,7.6);
}
\foreach \k in {12.5,14.5,15.5} {
\fill[color=bead] (3,\k) circle (10pt);}
\foreach \k in {10.5,11.5,12.5,13.5,14.5,15.5} {
\fill[color=bead] (4,\k) circle (10pt);}
\foreach \k in {8.5,10.5,11.5,12.5,13.5,14.5,15.5} {
\fill[color=bead] (5,\k) circle (10pt);}
\node at (7,13.5){$\xrightarrow{\Phi_2}$}; 
\draw(9.5,16)--(12.5,16); \draw(13.5,16)--(16.5,16); 
\foreach \k in {10,11,12,14,15,16} {
\fill [color=brown] (\k-.1,12)--(\k-.1,16) -- (\k+0.1,16)-- (\k+.1,12)--(\k-.1,12);
\fill[color=brown](\k-.1,16.1)--(\k+.1,16.1)--(\k+.1,16.3)--(\k-.1,16.3)--(\k-.1,16.1);
\fill[color=brown](\k-.1,11.9)--(\k+.1,11.9)--(\k+.1,11.7)--(\k-.1,11.7)--(\k-.1,11.9);
\fill[color=brown](\k-.1,11.6)--(\k+.1,11.6)--(\k+.1,11.4)--(\k-.1,11.4)--(\k-.1,11.6);
}
\foreach \k in {14.5,15.5} {
\fill[color=bead] (10,\k) circle (10pt);}
\foreach \k in {13.5,14.5,15.5} {
\fill[color=bead] (11,\k) circle (10pt);}
\foreach \k in {12.5,13.5,14.5,15.5} {
\fill[color=bead] (12,\k) circle (10pt);}
\foreach \k in {15.5} {
\fill[color=bead] (14,\k) circle (10pt);}
\foreach \k in {13.5,14.5,15.5} {
\fill[color=bead] (15,\k) circle (10pt);}
\foreach \k in {13.5,14.5,15.5} {
  \fill[color=bead] (16,\k) circle (10pt);}
\foreach \a in {3,4,5} {
\foreach \k in {8.5,...,15.5}{
\node at (\a,\k)[color=bead]{$-$};};};
\foreach \a in {10,11,12,14,15,16} {
\foreach \k in {12.5,...,15.5}{
\node at (\a,\k)[color=bead]{$-$};};};
\node at (17,13.5){,};
\end{tikzpicture}

Thus we have listed the elements of two of the $\ap_3$-equivalence classes of $\A_e^2$; equivalently we've generated all the multipartitions corresponding to Specht modules in two blocks, the first for the algebra $\mathcal{H}_{2,14}(q,{\bf Q})$ where $1+q+q^2=0$ and ${\bf Q}=(q^7,q^9) = (q,1)$; and the second for the algebra $\mathcal{H}_{2,11}(q,{\bf Q})$ where $1+q+q^2=0$ and ${\bf Q}=(q^9,q^7)=(1,q)$. 
\end{ex}

\begin{corollary}
Let $(\bla,\bs) \in \A_e^r$. Then the map $ \eta\circ \Phi_r: \A_e^r \rightarrow \A_r^e$ is a bijection between the $\ap_e$-equivalence class of $(\bla,\bs)$ and the $\ap_r$-equivalence class of $\eta(\Phi_r(\bla,\bs))$.
\end{corollary}

This corollary gives a natural bijection between Specht modules indexed by $r$-multipartitions in a block of one Ariki-Koike algebra (which has quantum characteristic $e$) and Specht modules indexed by $e$-multipartitions in a block of another Ariki-Koike algebra (which has quantum characteristic $r$). In general, there is no algebraic equivalence between the blocks. 

\begin{ex}
  Let $e=r=3$ and $\mc=(0,0,0)$. Let $\R$ be the $\ap_3$-equivalence class of $\A^3_3$ consisting of the multipartitions
  \begin{align*}
  & (((1),\emp,\emp)),\mc), &  & ((\emp,(1),\emp),\mc) &  &  ((\emp,\emp,(1)),\mc). 
  \end{align*}
  Applying the map $\eta \circ \Phi_3$ to the each of the elements of $\R$ respectively gives a $\ap_3$-equivalence class $\R'$ consisting of the multipartitions
\begin{align*}
  & (((2),\emp,\emp)),(1,0,-1)), &  & (((1),\emp,(1)),(1,0,-1)) &  &  ((\emp,\emp,(1^2)),(1,0,-1)). 
\end{align*}
Following the notation of Section~\ref{Blocks}, we let $\tilde{\R}$ and $\tilde{\R}'$ denote the blocks of the respective Ariki-Koike algebras corresponding to $\R$ and $\R'$. 
Then (up to isomorphism) there is one simple module lying in the block $\tilde{\R}$, but two simple modules lying in the block $\tilde{\R}'$. 
(To see how to determine which multipartitions index simple modules, we refer the reader to~\cite[Definition~2.23]{Mathas:AKSurvey}.) 
\end{ex}

It would be interesting to consider when some kind of algebraic equivalence does occur.
In particular, we would like to know how our bijection relates to the level-rank duality proposed by Chuang and Miyachi~\cite{CM} and studied in~\cite[Section~6.3]{SVV}. 
From a computational point of view, our bijection can come in useful when trying to generate the multipartitions in a block.  

\begin{ex}
Take $e=3$ and $r=2$ and let $\bla=((2,2,1,1),(3,1^5))$ and $\mc=(1,0)$.
Suppose that we want to generate the $\sd$-equivalence class of $\bla$.
Choose $\bs=(7,9)$ so that $(\bla,\bs)$ has abacus configuration 
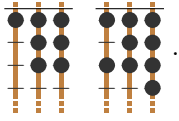
\begin{center} 
\begin{tikzpicture}\tikzset{yscale=0.3,xscale=0.3}
\foreach \k in {0,1,2,4,5,6} {
\fill [color=brown] (\k-.1,3)--(\k-.1,7) -- (\k+0.1,7)-- (\k+.1,3)--(\k-.1,3);
\fill[color=brown](\k-.1,7.1)--(\k+.1,7.1)--(\k+.1,7.3)--(\k-.1,7.3)--(\k-.1,7.1);
\fill[color=brown](\k-.1,2.9)--(\k+.1,2.9)--(\k+.1,2.7)--(\k-.1,2.7)--(\k-.1,2.9);
\fill[color=brown](\k-.1,2.6)--(\k+.1,2.6)--(\k+.1,2.4)--(\k-.1,2.4)--(\k-.1,2.6);
}
\foreach \k in {6.5} {
\fill[color=bead] (0,\k) circle (10pt);}
\foreach \k in {4.5,5.5,6.5} {
\fill[color=bead] (1,\k) circle (10pt);}
\foreach \k in {4.5,5.5,6.5} {
\fill[color=bead] (2,\k) circle (10pt);}
\foreach \k in {4.5,6.5} {
\fill[color=bead] (4,\k) circle (10pt);}
\foreach \k in {4.5,5.5,6.5} {
\fill[color=bead] (5,\k) circle (10pt);}
\foreach \k in {3.5,4.5,5.5,6.5} {
  \fill[color=bead] (6,\k) circle (10pt);}
\draw(-.5,7)--(2.5,7); \draw(3.5,7)--(6.5,7);
\foreach \a in {0,1,2,4,5,6} {
\foreach \k in {3.5,...,6.5}{
\node at (\a,\k)[color=bead]{$-$};};};
\node at (7,5){.};
\end{tikzpicture}
\end{center}
As we have seen above, $\eta(\Phi_2((\bla,\bs)))= (((2),\emp,\emp),(3,6,7))$ and there is a bijection between the multipartitions in the block of $(((2,\emp,\emp))$ with quantum characteristic 2 and multicharge $(1,0,1)$ and the multipartitions in the block we are trying to generate. However, generating multipartitions in $\La^{(3)}_2$ is easier than generating multipartitions in $\La^{(2)}_{14}$.  
\end{ex}

\section{Rouquier multipartitions} \label{RouquierMP}
\subsection{Rouquier multipartitions and Uglov's map}
Suppose that $(\la,s)\in \A_e$ and that $$\eta(\la,s) = ((\rho_0,\rho_1,\ldots,\rho_{e-1}),(t_0,t_1,\ldots,t_{e-1})).$$ 
We say that $(\la,s)$ is a Rouquier partition if 
\[\wt(\la) \leq t_{i+1}-t_i+1 \quad \text{for all} \quad 0 \leq i < e-1.\] 

Intuitively, we can think of $(\la,s)$ as being a Rouquier partition if the abacus configuration of its core (with respect to $s$) has the property that the difference between the number of beads on runner $i+1$ and runner $i$ is at least $\wt(\la)-1$. However this is not well-defined, since our convention is to have infinitely many beads on each runner. We say that $(\bla,\bs) \in \A_e^r$ is a Rouquier multipartition if each component $(\la^{(k)},s_k)$ is a Rouquier partition for $1 \leq k \leq r$. Say that a $\ap_e$-equivalence class $\R \subset \A_e^{r}$ is a Rouquier block if $(\bla,\bs)$ is a Rouquier multipartition for all $(\bla,\bs) \in \R$. 

On one hand, this is a perfectly reasonable definition. On the other hand, it is not so satisfactory because without writing out all the elements of $\R$ it seems difficult to tell whether or not it is a Rouquier block.

\begin{ex}
Suppose that $(\bla,\bs)$ and $(\bmu,\bs)$ have abacus configurations respectively given by 
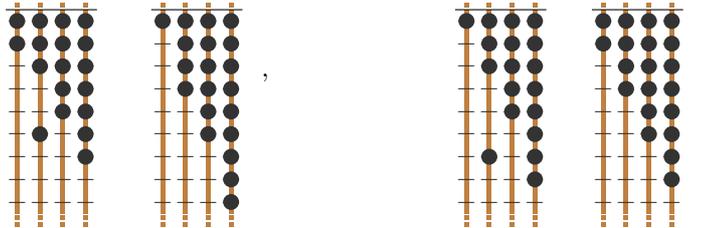
\begin{center}
\begin{tikzpicture}\tikzset{yscale=0.3,xscale=0.3}
\foreach \k in {0,1,2,3} {
\fill [color=brown] (\k-.1,-2)--(\k-.1,7) -- (\k+0.1,7)-- (\k+.1,-2)--(\k-.1,-2);
\fill[color=brown](\k-.1,7.1)--(\k+.1,7.1)--(\k+.1,7.3)--(\k-.1,7.3)--(\k-.1,7.1);
\fill[color=brown](\k-.1,-2.1)--(\k+.1,-2.1)--(\k+.1,-2.3)--(\k-.1,-2.3)--(\k-.1,-2.3);
\fill[color=brown](\k-.1,-2.4)--(\k+.1,-2.4)--(\k+.1,-2.6)--(\k-.1,-2.6)--(\k-.1,-2.6);
}
\foreach \k in {5.5,6.5} {
\fill[color=bead] (0,\k) circle (10pt);}
\foreach \k in {1.5,4.5,5.5,6.5} {
\fill[color=bead] (1,\k) circle (10pt);}
\foreach \k in {2.5,3.5,4.5,5.5,6.5} {
\fill[color=bead] (2,\k) circle (10pt);}
\foreach \k in {6.5,5.5,4.5,3.5,2.5,1.5,0.5} {
\fill[color=bead] (3,\k) circle (10pt);}
\draw(-.5,7)--(3.5,7); 
\foreach \a in {0,1,2,3} {
\foreach \k in {-1.5,...,6.5}{
\node at (\a,\k)[color=bead]{$-$};};};
\end{tikzpicture}
\; \,
\begin{tikzpicture}\tikzset{yscale=0.3,xscale=0.3}
\foreach \k in {0,1,2,3} {
\fill [color=brown] (\k-.1,-2)--(\k-.1,7) -- (\k+0.1,7)-- (\k+.1,-2)--(\k-.1,-2);
\fill[color=brown](\k-.1,7.1)--(\k+.1,7.1)--(\k+.1,7.3)--(\k-.1,7.3)--(\k-.1,7.1);
\fill[color=brown](\k-.1,-2.1)--(\k+.1,-2.1)--(\k+.1,-2.3)--(\k-.1,-2.3)--(\k-.1,-2.1);
\fill[color=brown](\k-.1,-2.4)--(\k+.1,-2.4)--(\k+.1,-2.6)--(\k-.1,-2.6)--(\k-.1,-2.6);
}
\foreach \k in {6.5} {
\fill[color=bead] (0,\k) circle (10pt);}
\foreach \k in {3.5,4.5,5.5,6.5} {
\fill[color=bead] (1,\k) circle (10pt);}
\foreach \k in {1.5,2.5,3.5,4.5,5.5,6.5} {
\fill[color=bead] (2,\k) circle (10pt);}
\foreach \k in {6.5,5.5,4.5,3.5,2.5,1.5,0.5,-0.5,-1.5} {
\fill[color=bead] (3,\k) circle (10pt);}
\draw(-.5,7)--(3.5,7); 
\node at (4.5,4) {,};
\foreach \a in {0,1,2,3} {
\foreach \k in {-1.5,...,6.5}{
\node at (\a,\k)[color=bead]{$-$};};};
\end{tikzpicture}
\qquad \qquad \qquad 
\begin{tikzpicture}\tikzset{yscale=0.3,xscale=0.3}
\foreach \k in {0,1,2,3} {
\fill [color=brown] (\k-.1,-2)--(\k-.1,7) -- (\k+0.1,7)-- (\k+.1,-2)--(\k-.1,-2);
\fill[color=brown](\k-.1,7.1)--(\k+.1,7.1)--(\k+.1,7.3)--(\k-.1,7.3)--(\k-.1,7.1);
\fill[color=brown](\k-.1,-2.1)--(\k+.1,-2.1)--(\k+.1,-2.3)--(\k-.1,-2.3)--(\k-.1,-2.3);
\fill[color=brown](\k-.1,-2.4)--(\k+.1,-2.4)--(\k+.1,-2.6)--(\k-.1,-2.6)--(\k-.1,-2.6);
}
\foreach \k in {6.5} {
\fill[color=bead] (0,\k) circle (10pt);}
\foreach \k in {0.5,4.5,5.5,6.5} {
\fill[color=bead] (1,\k) circle (10pt);}
\foreach \k in {2.5,3.5,4.5,5.5,6.5} {
\fill[color=bead] (2,\k) circle (10pt);}
\foreach \k in {6.5,5.5,4.5,3.5,2.5,1.5,0.5,-0.5} {
\fill[color=bead] (3,\k) circle (10pt);}
\draw(-.5,7)--(3.5,7); 
\foreach \a in {0,1,2,3} {
\foreach \k in {-1.5,...,6.5}{
\node at (\a,\k)[color=bead]{$-$};};};
\end{tikzpicture}
\;\,
\begin{tikzpicture}\tikzset{yscale=0.3,xscale=0.3}
\foreach \k in {0,1,2,3} {
\fill [color=brown] (\k-.1,-2)--(\k-.1,7) -- (\k+0.1,7)-- (\k+.1,-2)--(\k-.1,-2);
\fill[color=brown](\k-.1,7.1)--(\k+.1,7.1)--(\k+.1,7.3)--(\k-.1,7.3)--(\k-.1,7.1);
\fill[color=brown](\k-.1,-2.1)--(\k+.1,-2.1)--(\k+.1,-2.3)--(\k-.1,-2.3)--(\k-.1,-2.1);
\fill[color=brown](\k-.1,-2.4)--(\k+.1,-2.4)--(\k+.1,-2.6)--(\k-.1,-2.6)--(\k-.1,-2.6);
}
\foreach \k in {5.5,6.5} {
\fill[color=bead] (0,\k) circle (10pt);}
\foreach \k in {3.5,4.5,5.5,6.5} {
\fill[color=bead] (1,\k) circle (10pt);}
\foreach \k in {1.5,2.5,3.5,4.5,5.5,6.5} {
\fill[color=bead] (2,\k) circle (10pt);}
\foreach \k in {6.5,5.5,4.5,3.5,2.5,1.5,0.5,-0.5} {
\fill[color=bead] (3,\k) circle (10pt);}
\draw(-.5,7)--(3.5,7); 
\node at (4.5,4) {.};
\foreach \a in {0,1,2,3} {
\foreach \k in {-1.5,...,6.5}{
\node at (\a,\k)[color=bead]{$-$};};};
\end{tikzpicture}
\end{center}
Then $(\bla,\bs) \ap_4 (\bmu,\bs)$. (One may see this by comparing the $e$-residue diagrams.) However $(\bla,\bs)$ is a Rouquier multipartition while $(\bmu,\bs)$ is not. 
\end{ex}

We note a slight disparity between the cases that $r=1$ and $r>1$. If $r=1$ and $(\la,s) \ap_e (\mu,s)$ then $(\la,s)$ is a Rouquier partition if and only if $(\mu,s)$ is a Rouquier partition. However as we have seen above, if $r>1$ then it is perfectly possible to have $(\bla,\bs) \ap_e (\bmu,\bs)$ where $(\bla,\bs)$ is a Rouquier multipartition but $(\bmu,\bs)$ is not. 

In fact, it is easier than we first believed to check if a $\ap_e$-equivalence class $\R$ is a Rouquier block. For $\bla \in \R$, let $\hk(\bla)=\sum_{k=1}^r \wt(\la^{(k)})$ 
and let $\hk(\R)=\max\{\hk(\bla) \mid \bla \in \R\}$. Let $\R^{\circ} = \{\bla \in \R \mid \hk(\bla)=\hk(\R)\}$. 
In~\cite{Fayers:Cores}, Fayers defines a core block of the Ariki-Koike algebra to be a block in which none of the associated multipartitions have a removable $e$-rim hook. To each block $\R \subset \A^{r}_e$ there is an associated core block $\mathscr{C}$, and the subset $\R^\circ$ is precisely the set of multipartitions obtained by adding $\hk(\R)$ $e$-rim hooks to any of the multipartitions in $\mathscr{C}$. The following result can be deduced from results of Dell'Arciprete~\cite[Proposition~4.8]{DA} and appears explictly in her thesis~\cite{DAThesis}. 

\begin{lemma}
Suppose that $\R$ is a $\ap_e$-equivalence class of $\R$. Then $\R$ is a Rouquier block if and only if every multipartition $\bla \in \R^\circ$ is a Rouquier multipartition.
\end{lemma}

Thus to see if a block is a Rouquier block, we just need to look at the associated core block.  
To further investigate Rouquier blocks, we use the map $\Phi_r$ introduced in the last section. We begin with one more definition. If $(\la,s)\in \A_e$ is such that $\eta(\la,s) = (\brho,\bt)$ then we say that $(\la,s)$ is an $r$-Rouquier partition if
\[\wt(\la) \leq t_{i+1}-t_i+r\quad  \text{for all} \quad 0 \leq i < e-1.\] In this case, every $(\mu,s)\in \A_e$ with $(\mu,s) \ap_e (\la,s)$ is also a $r$-Rouquier partition and we say that the $\ap_e$-equivalence class of $(\la,s)$ is a $r$-Rouquier block. 

\begin{lemma} \label{NotRR}
Suppose that $(\bla,\bs) \in \A_e^r$ is not a Rouquier multipartition. Then $\Psi_r(\bla,\bs) \in \A_e$ is not an $r$-Rouquier partition.
\end{lemma}

\begin{proof}
  Suppose that $(\bla,\bs)$ is not a Rouquier multipartition. Then there exists $1 \leq k \leq r$ such that $(\la^{(k)},s_k)$ is not a Rouquier partition. Write $\la^{(k)}=\la$. If we suppose that $\eta(\la,s_k) = ((\rho_0,\rho_1,\ldots,\rho_{e-1}),(t_0,t_1,\ldots,t_{e-1}))$ then there exists $0 \leq i <e-1$ such that $\wt(\la)>t_{i+1}-t_i +1$. Take $\mu \sim_e \la$ to be the partition with $\eta(\mu,s_k) = (\emp,\ldots,\emp,(\wt(\la)),\emp,\ldots,\emp),(t_0,t_1,\ldots,t_{e-1}))$ where $(\wt(\la))$ occurs in position $i$. Let $\bmu$ be the multipartition obtained from $\bla$ by replacing $\la$ with $\mu$ in component $k$. Then $\Psi_r(\bla,\bs) \ap_e \Psi_r(\bmu,\bs)$ by Proposition~\ref{BlockPreserve} so $\Psi_r(\bla,\bs)$ is an $r$-Rouquier partition if and only if $\Psi_r(\bmu,\bs)$ is an $r$-Rouquier partition.

Our aim then is to show that $\Psi_r(\bmu,\bs)$ is not an $r$-Rouquier partition. Set $\wt(\la)=\bar{w}$ and $t_{i+1}-t_i=\bar{d}$ so that by assumption $\bar{w}>\bar{d}+1$. Suppose that $\eta(\Psi_r(\bmu,\bs)) = ((\sigma_0,\sigma_1,\ldots,\sigma_{e-1}),(r_0,r_1,\ldots,r_{e-1}))$. Let $w=|\sigma_i|+|\sigma_{i+1}|$ and $d=r_{i+1}-r_i$ so that $\wt(\Psi_r(\bmu,\bs)) \geq w$.  We will show that $w>d+r$.

Let $B$ be the $\beta$-set corresponding to $\Psi_r(\bmu,\bs)$. Let $M$ be maximal such that $M \equiv r-k \mod r$ and $Me+i \in B$. Note that the condition $\bar{w}>\bar{d}+1$ implies that $(M-r)e+i+1, Me+i+1 \notin B$. Let $\Z^{\vee (r-k)} = \{ m \in \Z \mid m \not \equiv r-k\mod r\}$. Define
  \begin{align*}
S & = \{m \in\Z^{\vee( r-k)} \mid me+i \notin B \text{ and } me+i+1 \in B\}, & 
T & = \{m \in\Z^{\vee( r-k)} \mid me+i \in B \text{ and } me+i+1 \notin B\}, \\
U & = \{m \in\Z^{\vee( r-k)} \mid me+i \in B \text{ and } me+i+1 \in B\}, &
V & = \{m \in\Z^{\vee( r-k)} \mid me+i \notin B \text{ and } me+i+1 \notin B\}.
  \end{align*}
  Now set
  \[s_1 = \#\{m \in S \mid m<M-r\}, \qquad s_2 = \#\{m \in S \mid M-r<m<M\}, \qquad s_3 = \#\{m \in S \mid M<m\},\]
  and define $t_1,t_2,t_3,u_2,u_3,v_1$ and $v_2$ analogously. Although the sets $U$ and $V$ are infinite, $u_2,\,u_3,\,v_1$ and $v_2$ all measure finite subsets and we do not define $u_1$ or $v_3$. Now
  \begin{align*}
d & = \#\{m \in\Z \mid me+i \notin B \text{ and } me+i+1 \in B\} - \#\{m \in\Z \mid me+i \in B \text{ and } me+i+1 \notin B\} \\
& = \#\{m \in\Z^{\vee(r-k)} \mid me+i \notin B \text{ and } me+i+1 \in B\} -  \#\{m \in\Z^{\vee( r-k)} \mid me+i \in B \text{ and } me+i+1 \notin B\} \\
& \qquad + \#\{m \equiv r-k \mod r \mid me+i \notin B \text{ and } me+i+1 \in B\}\\ & \qquad -  \#\{m \equiv r-k \mod r \mid me+i \in B \text{ and } me+i+1 \notin B\} \\
&=s_1+s_2+s_3-t_1-t_2-t_3+\bar{d}
  \end{align*}
so that
  \begin{align*}
    d+r&=s_1+s_2+s_3-t_1-t_2-t_3+\bar{d}+r\\
    &= s_1+s_2+s_3-t_1-t_2-t_3 + \bar{d} + s_2+t_2+u_2+v_2 +1 \\
    & \leq s_1+s_2+s_3+\bar{d} +s_2+u_2+v_2 +1\\
    & < s_1+s_2+s_3+\bar{w}+s_2+u_2+v_2.
  \end{align*}
 Now we give a lower bound for $w = |\sigma_i|+|\sigma_{i+1}|$. Recalling the way in which we chose $M$, we have
  \begin{align*}
    |\sigma_i| &= \#\{(m_1,m_2) \in \Z \times \Z \mid m_1 < m_2 \text{ and } m_1e+i \notin B, \, m_2e+i \in B\} \\
    & \geq (t_3+u_3+1)(s_1+s_2+v_1+v_2+\bar{w}) + (t_2+u_2)(\bar{w}+s_1+v_1), \\
    |\sigma_{i+1}| &= \#\{(m_1,m_2) \in \Z \times \Z \mid m_1 < m_2 \text{ and } m_1e+i+1 \notin B, \, m_2e+i+1 \in B\} \\
    & \geq (s_3+u_3)(2+t_1+t_2+v_1+v_2) + (s_2+u_2)(1+t_1+v_1).
   \end{align*}
Hence    
\begin{align*} 
    w &\geq (t_3+u_3+1)(s_1+s_2+v_1+v_2+\bar{w}) + (t_2+u_2)(\bar{w}+s_1+v_1) \\
    & \qquad + (s_3+u_3)(2+t_1+t_2+v_1+v_2) + (s_2+u_2)(1+t_1+v_1) \\ 
    & \geq s_1+s_2+v_2 + \bar{w}+s_3 +s_2+u_2 
    \end{align*}
    so that $w>d+r$ as required. 
  \end{proof}

Lemma~\ref{NotRR} shows that if $(\la,s) \in \A_e$ is an $r$-Rouquier partition then $\Phi_r(\la,s) \in \A_e^r$ is a Rouquier multipartition. Since all elements in the same block as $(\la,s)$ are also $r$-Rouquier partitions, the next result follows by applying Corollary~\ref{DU}.

\begin{corollary} \label{RR}
Suppose that $\Ra \subset \A_e$ is an $r$-Rouquier block. Then $\Phi_r(\Ra) \subset \A^r_e$ is a disjoint union of Rouquier blocks.  
\end{corollary}

We can therefore generate blocks of Rouquier multipartitions by applying the map $\Phi_r$ to an $r$-Rouquier block. Unfortunately it is not true that every such block can be generated in this way. 

\begin{ex} Suppose that $r=2$ and $e=6$. Let $\bs=(30,30)$. Then the block containing the multipartition $(\bla,\bs)$ below is a Rouquier block. However, there does not exist $\bs' \in \Z^2$ such that $s_k \equiv s'_i \mod 6$ for $k=1,2$ and $\Psi_r(\bla,\bs')$ lies in a $2$-Rouquier block. 

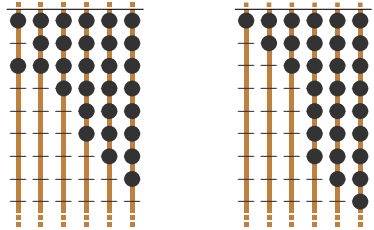
\begin{center}
\begin{tikzpicture}\tikzset{yscale=0.3,xscale=0.3}
\foreach \k in {0,1,2,3,4,5} {
\fill [color=brown] (\k-.1,-2)--(\k-.1,7) -- (\k+0.1,7)-- (\k+.1,-2)--(\k-.1,-2);
\fill[color=brown](\k-.1,7.1)--(\k+.1,7.1)--(\k+.1,7.3)--(\k-.1,7.3)--(\k-.1,7.1);
\fill[color=brown](\k-.1,-2.1)--(\k+.1,-2.1)--(\k+.1,-2.3)--(\k-.1,-2.3)--(\k-.1,-2.1);
\fill[color=brown](\k-.1,-2.4)--(\k+.1,-2.4)--(\k+.1,-2.6)--(\k-.1,-2.6)--(\k-.1,-2.6);
}
\foreach \k in {4.5,6.5} {
\fill[color=bead] (0,\k) circle (10pt);}
\foreach \k in {4.5,5.5,6.5} {
\fill[color=bead] (1,\k) circle (10pt);}
\foreach \k in {3.5,4.5,5.5,6.5} {
\fill[color=bead] (2,\k) circle (10pt);}
\foreach \k in {1.5,2.5,3.5,4.5,5.5,6.5} {
\fill[color=bead] (3,\k) circle (10pt);}
\foreach \k in {0.5,1.5,2.5,3.5,4.5,5.5,6.5} {
\fill[color=bead] (4,\k) circle (10pt);}
\foreach \k in {-.5,.5,1.5,2.5,3.5,4.5,5.5,6.5} {
\fill[color=bead] (5,\k) circle (10pt);}

\draw(-.5,7)--(5.5,7); 
\foreach \a in {0,1,2,3,4,5} {
\foreach \k in {-1.5,...,6.5}{
\node at (\a,\k)[color=bead]{$-$};};};

\foreach \k in {10,11,12,13,14,15} {
\fill [color=brown] (\k-.1,-2)--(\k-.1,7) -- (\k+0.1,7)-- (\k+.1,-2)--(\k-.1,-2);
\fill[color=brown](\k-.1,7.1)--(\k+.1,7.1)--(\k+.1,7.3)--(\k-.1,7.3)--(\k-.1,7.2);
\fill[color=brown](\k-.1,-2.1)--(\k+.1,-2.1)--(\k+.1,-2.3)--(\k-.1,-2.3)--(\k-.1,-2.1);
\fill[color=brown](\k-.1,-2.4)--(\k+.1,-2.4)--(\k+.1,-2.6)--(\k-.1,-2.6)--(\k-.1,-2.6);
}
\foreach \k in {6.5} {
\fill[color=bead] (10,\k) circle (10pt);}
\foreach \k in {5.5,6.5} {
\fill[color=bead] (11,\k) circle (10pt);}
\foreach \k in {4.5,5.5,6.5} {
\fill[color=bead] (12,\k) circle (10pt);}
\foreach \k in {0.5,1.5,2.5,3.5,4.5,5.5,6.5} {
\fill[color=bead] (13,\k) circle (10pt);}
\foreach \k in {-.5,0.5,1.5,2.5,3.5,4.5,5.5,6.5} {
\fill[color=bead] (14,\k) circle (10pt);}
\foreach \k in {-1.5,-.5,.5,1.5,2.5,3.5,4.5,5.5,6.5} {
\fill[color=bead] (15,\k) circle (10pt);}
\draw(9.5,7)--(15.5,7); 
\foreach \a in {10,11,12,13,14,15} {
\foreach \k in {-1.5,...,6.5}{
\node at (\a,\k)[color=bead]{$-$};};};

\end{tikzpicture}
\end{center}
To see this, note that if $\bs'\in \Z^2$ is such that $s_k \equiv s'_k \mod 6$ for $k=1,2$ then $$\eta(\Psi_r(\bla,\bs'))=((\sigma_0,\sigma_1,\ldots,\sigma_{e-1}),(t_0+3,t_0+5,t_0+7,t_0+13,t_0+15,t_0+17))$$ for some $t_0 \in \Z$. So for $\Psi_r(\bla,\bs')$ to lie in a $2$-Rouquier block, we need $\wt(\Psi_r(\bla,\bs')) \leq 4$. However if $s_1'\geq s_2'$ then $|\sigma_0|\geq 3$ and $|\sigma_1|,|\sigma_2|\geq 1$ and if $s_1'<s_2'$ then $|\sigma_0|\geq 2$ and $|\sigma_3|,|\sigma_4|,|\sigma_5|\geq 1$, so that in both cases $\wt(\Psi_r(\bla,\bs')) = \sum_{i=0}^5 |\sigma_i|\ge 5$.  

\end{ex}

If $r=e=2$, we have a converse to Corollary~\ref{RR}. We do not give a proof here as we hope to return to Rouquier blocks for the case that $e=r=2$ in a later paper. 

\begin{lemma}
Suppose that $e=r=2$ and that $(\bla,\bs) \in \A_2^2$ lies in a Rouquier block. Then there exists $\bs' \in \Z^2$ with $s_i \equiv s_i'$ for $i=1,2$ such that $\Psi_2(\bla,\bs')$ lies in a 2-Rouquier block.
\end{lemma}

We will obtain a general converse to Corollary~\ref{RR} as follows. First we introduce the notion of a `stretched block' and we show that all sufficiently stretched blocks $\R \subset \A_e^r$ have the property that $\Psi_r(\R)$ lies in an $r$-Rouquier block. Secondly we show that when we stretch a Rouquier block, we obtain a block which is in some sense equivalent to the original block.

\subsection{Stretching} \label{Stretching}
Let $\M=(M_0,M_1,\ldots,M_{e-1}) \in \Z^{e}$. Suppose that $(\la,s) \in \A_e$ where $$\eta(\la,s) = ((\rho_0,\rho_1,\ldots,\rho_{e-1}),(t_0,t_1,\ldots,t_{e-1})).$$ Define $\Str((\la,s);\M)$ to be the abacus configuration in $\A_e$ such that
\[\eta(\Str((\la,s);\M)) = ((\rho_0,\rho_1,\ldots,\rho_{e-1}),(t_0+M_0,t_1+M_1,\ldots,t_{e-1}+M_{e-1})).\]
In other words, the intuitive idea is that when we stretch $(\la,s)$, the $e$-quotient remains the same but we change the number of beads on each runner. We extend the definition of stretching to tuples of abacuses in the natural way. Suppose that $(\bla,\bs) \in \A_e^r$ and that $\M=(M_0,M_1,\ldots,M_{e-1}) \in \Z^{e}_{\geq 0}$. Then
\[\Str((\bla,\bs);\M)=((\mu^{(1)},\mu^{(2)},\ldots,\mu^{(r)}),(s'_1,s'_2,\ldots,s'_r))\]
where $(\mu^{(k)},s_k') = \Str((\la^{(k)},s_k);\M)$ for $1 \leq k \leq r$. 

\begin{ex}
Let $\M=(0,1,3)$. If we take the abacus configuration $(\bla,\bs)$ on the left and apply the stretching operation, we obtain the abacus configuration $\Str((\bla,\bs);M)$ on the right. 
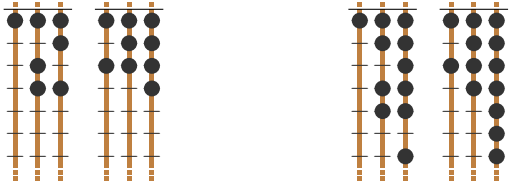
\begin{center} 
\begin{tikzpicture}\tikzset{yscale=0.3,xscale=0.3}
\foreach \k in {0,1,2,4,5,6} {
\fill [color=brown] (\k-.1,0)--(\k-.1,7) -- (\k+0.1,7)-- (\k+.1,0)--(\k-.1,0);
\fill[color=brown](\k-.1,7.1)--(\k+.1,7.1)--(\k+.1,7.3)--(\k-.1,7.3)--(\k-.1,7.1);
\fill[color=brown](\k-.1,-0.1)--(\k+.1,-0.1)--(\k+.1,-0.3)--(\k-.1,-0.3)--(\k-.1,-0.1);
\fill[color=brown](\k-.1,-0.4)--(\k+.1,-0.4)--(\k+.1,-0.6)--(\k-.1,-0.6)--(\k-.1,-0.4);
}
\foreach \k in {6.5} {
\fill[color=bead] (0,\k) circle (10pt);}
\foreach \k in {4.5,3.5,6.5} {
\fill[color=bead] (1,\k) circle (10pt);}
\foreach \k in {3.5,5.5,6.5} {
\fill[color=bead] (2,\k) circle (10pt);}
\foreach \k in {4.5,6.5} {
\fill[color=bead] (4,\k) circle (10pt);}
\foreach \k in {4.5,5.5,6.5} {
\fill[color=bead] (5,\k) circle (10pt);}
\foreach \k in {3.5,4.5,5.5,6.5} {
  \fill[color=bead] (6,\k) circle (10pt);}
\draw(-.5,7)--(2.5,7); \draw(3.5,7)--(6.5,7);
\foreach \a in {0,1,2,4,5,6} {
\foreach \k in {0.5,...,6.5}{
\node at (\a,\k)[color=bead]{$-$};};};
\end{tikzpicture}
\qquad \qquad \qquad 
\begin{tikzpicture}\tikzset{yscale=0.3,xscale=0.3}
\foreach \k in {0,1,2,4,5,6} {
\fill [color=brown] (\k-.1,0)--(\k-.1,7) -- (\k+0.1,7)-- (\k+.1,0)--(\k-.1,0);
\fill[color=brown](\k-.1,7.1)--(\k+.1,7.1)--(\k+.1,7.3)--(\k-.1,7.3)--(\k-.1,7.1);
\fill[color=brown](\k-.1,-0.1)--(\k+.1,-0.1)--(\k+.1,-0.3)--(\k-.1,-0.3)--(\k-.1,-0.1);
\fill[color=brown](\k-.1,-0.4)--(\k+.1,-0.4)--(\k+.1,-0.6)--(\k-.1,-0.6)--(\k-.1,-0.4);
}
\foreach \k in {6.5} {
\fill[color=bead] (0,\k) circle (10pt);}
\foreach \k in {5.5,2.5,3.5,6.5} {
\fill[color=bead] (1,\k) circle (10pt);}
\foreach \k in {0.5,2.5,3.5,4.5,5.5,6.5} {
\fill[color=bead] (2,\k) circle (10pt);}
\foreach \k in {4.5,6.5} {
\fill[color=bead] (4,\k) circle (10pt);}
\foreach \k in {3.5,4.5,5.5,6.5} {
\fill[color=bead] (5,\k) circle (10pt);}
\foreach \k in {0.5,1.5,2.5,3.5,4.5,5.5,6.5} {
  \fill[color=bead] (6,\k) circle (10pt);}
\draw(-.5,7)--(2.5,7); \draw(3.5,7)--(6.5,7);
\foreach \a in {0,1,2,4,5,6} {
\foreach \k in {0.5,...,6.5}{
\node at (\a,\k)[color=bead]{$-$};};};
\end{tikzpicture} \qquad

\end{center}

\end{ex}

We now notice that the stretching operation preserves blocks, hence if $\R$ is a block of $\A^r_e$, we may consider the block $\Str(\R;\M)$.

\begin{lemma} Let $\M=(M_0,M_1,\ldots,M_{e-1}) \in \Z^{e}_{\geq 0}$. 
Suppose that $(\bla,\bs),(\bmu,\bs') \in \A^r_e$. Then $(\bla,\bs) \ap_e (\bmu,\bs')$ if and only if $\Str((\bla,\bs);\M) \ap_e \Str((\bmu,\bs');\M)$. 
\end{lemma}

\begin{proof}
This follows from Proposition~\ref{GenBlock}. 
\end{proof}

We now show that if we take an arbitrary block $\R \subset \A^r_e$ and stretch it sufficiently then we obtain a block $\R'$ with the property that $\Psi_r(\R')$ lies in an $r$-Rouquier block. From Corollary~\ref{RR}, it follows immediately that $\R'$ is itself a Rouquier block.

\begin{lemma} \label{GGR} Let $M \gg 0$ and set $\M=(0,M,\ldots,(e-1)M) \in \Z^{e}_{\geq 0}$. Let $\R$ be a  $\ap_e$-equivalence class of  $\A^r_e$. Then $\Psi_r(\Str(\R;\M))$ lies in a Rouquier block, and hence in an $r$-Rouquier block.
\end{lemma}

\begin{proof}
Suppose that $(\bla,\bs) \in \R$. First note that if
\[\eta(\Psi_r(\bla,\bs)) = ((\rho_0,\rho_1,\ldots,\rho_{e-1}),(t_0,t_1,\ldots,t_{e-1}))\]
then
\[\eta(\Psi_r(\Str((\bla,\bs);\M))) = ((\rho_0,\rho_1,\ldots,\rho_{e-1}),(t_0,t_1+rM,\ldots,t_{e-1}+(e-1)rM)).\]
Set $T=\max\{t_i-t_{i+1} \mid 0 \leq i <e-1\}$ and choose $M$ such that $rM \geq T+ |\brho|$. If $0 \leq i < e-1$ then
\[t_{i+1} +(i+1)rM - (t_i +i rM) = rM-(t_i-t_{i+1}) \geq rM - T \geq |\brho|\]
so that $\Psi_r(\Str((\bla,\bs);\M))$ is a Rouquier partition, and hence an $r$-Rouquier partition. 
Applying this argument to every configuration in the block, we obtain the result. 
\end{proof}

\subsection{Equivalences} \label{Eqs} 
So far, our stretching operation has just defined a bijection between $\ap_e$-equivalence classes. For this to be of interest, we require that this bijection preserves some algebraic structure. Below, we recall Scopes' theorem on equivalences between blocks of the Hecke algebras of type $A$. One of the strengths of Scopes equivalence is that as well as showing certain blocks are Morita equivalent by defining functors between the module categories, it describes an explicit and very natural bijection between the Specht modules in each block which corresponds to the action of the functors. The notion of Scopes equivalence can be extended to Ariki-Koike algebras in a natural way. 

Before discussing Scopes equivalence, we introduce the notion of `decomposition equivalence'.  Suppose that $\R$ and $\R'$ are $\ap_e$-equivalence classes in $\A^{r}_e$ with $\widetilde{\R}$ a block of $\mathcal{H}$ and $\widetilde{\R'}$ a block of $\mathcal{H}'$. Then there exist $\bs,\bs' \in \Z^r$ such that $\R$ consists only of elements of the form $(\bla,\bs)$ and $\R'$ consists only of elements of the form $(\bla',\bs')$. Let $\mc$ (resp. $\mc'$) $\in I^r$ be such that $a_k \equiv s_k$ (resp. $a'_k \equiv s'_k$) $\mod e$ for $1 \leq k \leq r$. If $\Pi:\R \rightarrow \R'$ and $(\bla,\bs) \in \R$ with $\Pi(\bla,\bs)=(\bla',\bs')$, set $\Pi(\bla)=\bla'$. 
Using the notation above, we say that $\R$ and $\R'$ are decomposition equivalent if there is a bijection $\Pi:\R \rightarrow \R'$ such that for all $(\bmu,\bs), (\bla,\bs) \in \R$ we have
\begin{itemize}
\item $\bmu \in \Kla(\mc)$ if and only if $\Pi(\bmu) \in \Kla(\mc')$,
\item If $\bmu \in \Kla(\mc)$ then $[S^{\bla}:D^{\bmu}]_{\mathcal{H}} = [S^{\Pi(\bla)}:D^{\Pi(\bmu)}]_{\mathcal{H}'}$. 
\end{itemize}
We call the bijection $\Pi$ a decomposition equivalence map; essentially, it is set up so that it preserves block decomposition matrices. 

We now restrict our attention to $\hn$. Let $0 \leq i < e-1$. We define a bijective map $\Upsilon_i: \A_e \rightarrow \A_e$ as follows. Suppose $(\la,s) \in \A_e$ and suppose that $\eta(\la,s) = ((\rho_0,\rho_1,\ldots,\rho_{e-1}),(t_0,t_1,\ldots,t_{e-1}))$. Define $\Upsilon_i(\la,s)$ to be the abacus configuration $(\mu,s)$ such that $\eta(\mu,s)$ is obtained from $\eta(\la,s)$ by swapping $\rho_i$ and $\rho_{i+1}$ and swapping $t_i$ and $t_{i+1}$. 

Using Proposition~\ref{Whenr1} it is easy to see that if $\Ra$ is a $\ap_e$-equivalence class of $\A_e$ then $\Upsilon_i(\Ra)$ is also a $\ap_e$-equivalence class of $\A_e$. 

\begin{ex}
The bijection $\Upsilon_2$ swaps these two abacus configurations. The effect of $\Upsilon_2$ is simply that it interchanges runners $2$ and $3$. 

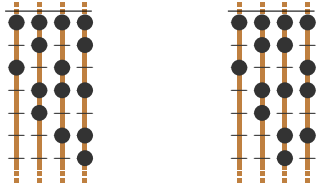
\begin{center}
\begin{tikzpicture}\tikzset{yscale=0.3,xscale=0.3}
\foreach \k in {0,1,2,3} {
\fill [color=brown] (\k-.1,0)--(\k-.1,6.5) -- (\k+0.1,6.5)-- (\k+.1,0)--(\k-.1,0);
\fill[color=brown](\k-.1,6.9)--(\k+.1,6.9)--(\k+.1,7.1)--(\k-.1,7.1)--(\k-.1,6.9);
\fill[color=brown](\k-.1,7.2)--(\k+.1,7.2)--(\k+.1,7.4)--(\k-.1,7.4)--(\k-.1,7.2);
\fill[color=brown](\k-.1,-0.1)--(\k+.1,-0.1)--(\k+.1,-0.3)--(\k-.1,-0.3)--(\k-.1,-0.1);
\fill[color=brown](\k-.1,-0.4)--(\k+.1,-0.4)--(\k+.1,-0.6)--(\k-.1,-0.6)--(\k-.1,-0.6);
}
\foreach \k in {4.5,6.5} {
\fill[color=bead] (0,\k) circle (10pt);}
\foreach \k in {2.5,3.5,5.5,6.5} {
\fill[color=bead] (1,\k) circle (10pt);}
\foreach \k in {1.5,3.5,4.5,6.5} {
\fill[color=bead] (2,\k) circle (10pt);}
\foreach \k in {0.5,1.5,3.5,5.5,6.5} {
\fill[color=bead] (3,\k) circle (10pt);}

\draw(-.5,7)--(3.3,7); 
\foreach \a in {0,1,2,3} {
\foreach \k in {0.5,...,6.5}{
\node at (\a,\k)[color=bead]{$-$};};};
\end{tikzpicture} \qquad \qquad 
\begin{tikzpicture}\tikzset{yscale=0.3,xscale=0.3}
\foreach \k in {0,1,2,3} {
\fill [color=brown] (\k-.1,0)--(\k-.1,6.5) -- (\k+0.1,6.5)-- (\k+.1,0)--(\k-.1,0);
\fill[color=brown](\k-.1,6.9)--(\k+.1,6.9)--(\k+.1,7.1)--(\k-.1,7.1)--(\k-.1,6.9);
\fill[color=brown](\k-.1,7.2)--(\k+.1,7.2)--(\k+.1,7.4)--(\k-.1,7.4)--(\k-.1,7.2);
\fill[color=brown](\k-.1,-0.1)--(\k+.1,-0.1)--(\k+.1,-0.3)--(\k-.1,-0.3)--(\k-.1,-0.1);
\fill[color=brown](\k-.1,-0.4)--(\k+.1,-0.4)--(\k+.1,-0.6)--(\k-.1,-0.6)--(\k-.1,-0.6);
}
\foreach \k in {4.5,6.5} {
\fill[color=bead] (0,\k) circle (10pt);}
\foreach \k in {2.5,3.5,5.5,6.5} {
\fill[color=bead] (1,\k) circle (10pt);}
\foreach \k in {1.5,3.5,4.5,6.5} {
\fill[color=bead] (3,\k) circle (10pt);}
\foreach \k in {0.5,1.5,3.5,5.5,6.5} {
\fill[color=bead] (2,\k) circle (10pt);}

\draw(-.5,7)--(3.3,7); 
\foreach \a in {0,1,2,3} {
\foreach \k in {0.5,...,6.5}{
\node at (\a,\k)[color=bead]{$-$};};};
\end{tikzpicture}
\end{center}
\end{ex}

We now describe Scopes' theorem. These results were initially proved by Scopes for the symmetric group algebras~\cite[Lemma~2.4 and Theorem~4.2]{Scopes} and then generalised to $\hn$~\cite{Jost}. 

\begin{theorem}[Scopes Equivalence] \label{Scopes} 
Suppose that $\Ra = \R(\bt;w)$ is a $\ap_e$-equivalence class of $\A_e$. Let $0 \leq i < e-1$ and suppose that $t_{i+1}-t_i \geq w$. 
\begin{itemize}
\item The map $\Upsilon_i$ restricts to a decomposition equivalence map on $\Ra$. 
\item The blocks $\widetilde{\Ra}$ and $\widetilde{\Upsilon_i(\Ra)}$ are Morita equivalent.
\end{itemize}
\end{theorem}

We note that decomposition equivalence and Morita equivalence are independent. 
We now look at a more trivial equivalence (in fact an equality). We define an invertible map $\upsilon: \A_e \rightarrow \A_e$ as follows. Suppose $(\la,s) \in \A_e$ and that $\eta(\la,s) = ((\rho_0,\rho_1,\ldots,\rho_{e-1}),(t_0,t_1,\ldots,t_{e-1}))$. Define $\upsilon(\la,s) = (\la,s+1)$ so that $\eta(\la,s+1)=(\rho_{e-1},\rho_0,\rho_1,\ldots,\rho_{e-2}),(t_{e-1}+1,t_0,t_1,\ldots,t_{e-2}))$. Note that the Hecke algebra $\mathcal{H}_{n}=\mathcal{H}_{1,n}(q,q^s)$ is independent of the choice of $s$. The following result then follows by definition. 

\begin{lemma} \label{Applyupsilon}
Suppose that $\Ra$ is a $\ap_e$-equivalence class of $\A_e$. Then $\widetilde{\Ra} = \widetilde{\upsilon(\Ra)}$.
\end{lemma}

This puts us in a position to reprove the well-known result that when $r=1$, any two Rouquier blocks of the same weight are both Morita equivalent and decomposition equivalent. We give a proof since we will use similar techniques to study the case that $r \geq 1$. 

\begin{lemma} \label{baseR}
Let $\bt=(t_0,t_1,\ldots,t_{e-1})$ and suppose that $\Ra=\Ra(\bt;w)$ is a Rouquier block in $\A_e$ where $w \geq 1$. Suppose that $0 \leq i \leq e-1$ and set $\bt'=(t_0,\ldots,t_{i-1},t_{i}+1,\ldots,t_{e-1}+1)$ and $\Ra'=\Ra(\bt';w)$. 
Then the blocks $\widetilde{\Ra}$ and $\widetilde{\Ra'}$ are Morita equivalent. Furthermore, the map $\Pi:\Ra \rightarrow \Ra'$ which sends $(\la,s) \in \Ra$ with $\eta(\la,s) = (\brho,\bt)$ to $(\la',s+e-i) \in \Ra'$ where $\eta(\la',s+e-i)=(\brho,\bt')$ is a decomposition equivalence map. 
\end{lemma}

\begin{proof} Write $\Ra \leftrightarrow \Ra'$ to indicate that Morita and decomposition equivalence both hold. 
Then we have 
\begin{align*}
\Ra(t_0,\ldots,t_{i-1},t_{i},\ldots,t_{e-1};w) & \leftrightarrow \Ra(t_{i}+1,\ldots,t_{e-1}+1,t_{0},\ldots,t_{i-1};w) && \text{repeatedly applying Lemma~\ref{Applyupsilon}} \\
 & \leftrightarrow \Ra(t_{0},\ldots,t_{i-1},t_{i}+1,\ldots,t_{e-1}+1;w) && \text{repeatedly applying Theorem~\ref{Scopes}}.
\end{align*}
Composing the decomposition maps completes the proof. Note that the condition for $\Ra$ to be a Rouquier block allows us to use Theorem~\ref{Scopes}. 
\end{proof}

\begin{corollary} \label{CEq}
Suppose that $\Ra=\Ra(\bt;w)$ and $\Ra'=(\bt';w)$ are Rouquier blocks in $\A_e$ of weight $w$. Then there exist $s,s' \in \Z$ such that 
\begin{align*}
\Ra & =\{(\la,s) \in \A_e \mid \eta(\la,s) = (\brho,\bt) \text{ where } |\brho|=w\}, \\
\Ra' & =\{(\la,s') \in \A_e \mid \eta(\la,s') = (\brho,\bt') \text{ where } |\brho|=w\}.
\end{align*}
Define $\Pi: \Ra \rightarrow \Ra'$ so that $\Pi(\brho,\bt)=\Pi(\brho,\bt')$ for $(\brho,\bt) \in \Ra$. 
Then the blocks $\widetilde{\Ra}$ and $\widetilde{\Ra'}$ are Morita equivalent and the map $\Pi$ is a decomposition equivalence map. 
\end{corollary}

\begin{proof}
Any two blocks of weight $0$ in $\A_e$ are Morita equivalent and decomposition equivalent (either apply Scopes equivalence or notice that a block of weight zero contains exactly one simple module) so suppose that $w\geq 1$. Applying Lemma~\ref{baseR} repeatedly, we see any Rouquier block of weight $w$ is both Morita equivalent and decomposition equivalent to the Rouquier block $\Ra(w-1,2(w-1),\ldots,e(w-1);w)$ with the decomposition equivalence map acting as described on the Specht modules. The result follows. 
\end{proof}
  
It is natural to ask to what extent the Scopes equivalences can be extended to $\h$. We generalise the map $\Upsilon_i:\A_e \rightarrow \A_e$ to a map $\Upsilon_i:\A^r_e \rightarrow \A^r_e$ in the natural way so that if $(\bla,\bs) \in \A^r_e$ is thought of as a tuple of abacuses then $\Upsilon_i$ acts as before on each component. We add a map $\Upsilon_{e-1}$ which acts on runners $0$ and $e-1$ in the natural way (by swapping the runners and then adding an extra bead to runner $0$). Now suppose that $\R$ is an $\ap_e$-equivalence class in $\A_e^r$ and $0 \leq i < e-1$ (resp. $i=e-1$) such that if $(\bla,\bs) \in \R$ with $\eta(\la^{(k)},s_k)=((\rho^k_0,\rho^k_1,\ldots,\rho^k_{e-1}),(t^k_0,t^k_1,\ldots,t_{e-1}^k))$ then $|\brho^k| \leq t^k_{i+1}-t^k_i$ for all $1 \leq k \leq r$ (resp. $|\brho^k| \leq t^k_{0}-t^k_{e-1}+1$ for all $1 \leq k \leq r$). We say that $\R$ and $\Upsilon_i(\R)$ are related by a Scopes move and we define an equivalence relation, Scopes equivalence, on the $\approx_e$-equivalence classes of $\A^r_e$ to be the relation generated by Scopes moves.   
Webster refers to a block that is Scopes equivalent to a Rouquier block as a RoCK block (see~\cite[Theorem~B]{Webster}). 

The following is a result of  Dell'Arciprete. 

\begin{theorem} [\cite{DA} Proposition~5.5] \label{DA}
Let $0 \leq i <e$. Suppose that $\R$ is an $\ap_e$-equivalence class in $\A_e^r$ such that $\Upsilon_i$ is a Scopes move. Then $\Upsilon_i$ restricts to a decomposition equivalence map from $\R$ to $\Upsilon(\R)$. 
\end{theorem}

In fact, Dell'Arciprete has a nicer combinatorial condition on the block than the one we give above; the current form (which does not introduce additional notation) is sufficient for our needs. 
The following theorem appeared as Conjecture~1 in the first draft of this paper and was subsequently proved by Webster. 

\begin{theorem} \cite{Webster} \label{C:Scopes}
Let $0 \leq i <e$. Suppose that $\R$ is an $\ap_e$-equivalence class in $\A_e^r$ such that $\Upsilon_i$ is a Scopes move. Then $\R$ and $\Upsilon_i(\R)$ are Morita equivalent. 
\end{theorem}

Now recall the map $\upsilon:\A_e \rightarrow \A_e$. We generalise it to a map $\upsilon:\A^r_e \rightarrow \A^r_e$ in the natural way by acting on each component. Suppose that $\bs=(s_1,s_2,\ldots,s_r) \in \Z^r$. Let ${\boldsymbol Q}=(q^{s_1},q^{s_2},\ldots,q^{s_r})$ and ${\boldsymbol Q'}=(q^{s_1+1},q^{s_2+1},\ldots,q^{s_r+1})$. For $n \geq 0$, the algebras $\mathcal{H}_{r,n}(q,{\boldsymbol Q})$ and $\mathcal{H}_{r,n}(q,{\boldsymbol Q'})$ are isomorphic. Hence we have an immediate analogue of  Lemma~\ref{Applyupsilon}. 

\begin{lemma} \label{Applyupsilon2}
Suppose that $\R$ is a $\ap_e$-equivalence class of $\A^r_e$. Then $\widetilde{\R} \cong \widetilde{\upsilon(\R)}$.
\end{lemma}

Our next step is to prove that if we stretch a Rouquier block then the block we obtain is decomposition and Morita equivalent. The proof of Corollary~\ref{CEq} split into two parts depending on whether or not the block had weight 0. Our proof splits up similarly, depending of whether or not the block is a core block.

Recall that Fayers~\cite{Fayers:Cores} defined a core block to be a block $\R \subset \A^r_e$ such that if $(\bla,\bs) \in \R$ then no component of $\bla$ has any $e$-rim hooks. The following observation follows from the definition of a Rouquier block. 

\begin{lemma} \label{IfNotCore}
Suppose that $\R$ is a Rouquier block in $\A^r_e$ and that $(\bla,\bs) \in \R$ has an $e$-rim hook. If $\eta(\la^{(k)},s_k) = (\brho^k,\bt^k)$ 
for $1 \leq k \leq r$ then $t^k_{i+1} - t^k_i \geq 0$ for all $1 \leq k \leq r$ and all $0 \leq i < e-1$. 
\end{lemma}

\begin{proof}
For each $1 \le k \le r$, there is a multipartition $(\bmu,\bs)$ with $(\bla,\bs) \xrightarrow{e}_1 (\bmu,\bs)$ and $\wt(\mu^{(k)}) \geq 1$; moreover $\eta(\mu^{(k)},s_k) = ({\boldsymbol \sigma}^k,\bt^k)$ for some ${\boldsymbol \sigma}^k \in \La^{(e)}$. Since $(\bmu,\bs)$ is also a Rouquier multipartition, the result follows. 
  \end{proof}

\begin{lemma} \label{Drop}
Suppose that $\R$ is a Rouquier block in $\A^r_e$ which is not a core block. Let $(\bla,\bs) \in \R$ and for $1 \leq k \leq r$, suppose that  $\eta(\la^{(k)},s_k) = (\brho^k,\bt^k)$.
Then \[\sum_{k=1}^r t^k_{i+1} - \sum_{k=1}^r t^k_i \geq 0\]
for all $0 \leq i <e-1$. 
\end{lemma}

\begin{proof}
First note that if 
$(\bmu,\bs) \in \R$ and $\eta(\mu^{(k)},s_k)=({\boldsymbol \sigma}^k,{\boldsymbol r}^k)$ 
 for $1 \leq k \leq r$, then by Proposition~\ref{GenBlock}, 
\[\sum_{k=1}^r t^k_{i} = \sum_{k=1}^r r^k_i\]
for all $0 \leq i \leq e-1$. So we may assume that $\bla$ has an $e$-rim hook. By Lemma~\ref{IfNotCore}, $t^k_{i+1} \geq t^k_i$ for all $1 \leq k \leq r$ and all $0 \leq i < e-1$. The result follows. 
\end{proof}

\begin{lemma} \label{NotCoreM1}
Suppose that $\R$ is a Rouquier block in $\A^r_e$ which is not a core block. Let $(\bla,\bs) \in \R$ and for $1 \leq k \leq r$, suppose that $\eta(\la^{(k)},s_k) = (\brho,\bt)$. 
Suppose $1 \leq k \leq r$ and $0 \leq i < i' \leq e-1$. Then $t^k_{i'}-t^k_i \geq -1$. 
\end{lemma}

\begin{proof}
If $\bla$ has an $e$-rim hook, the result follows from Lemma~\ref{IfNotCore} so assume that $\bla$ has no $e$-rim hooks. Suppose that there exist $k,i,i'$ as above which contradict the lemma. By Lemma~\ref{Drop} 
\[\sum_{l=1}^r t^l_{i'} - \sum_{l=1}^r t^l_{i} \geq 0\]
so in particular there exists $1 \leq k' \leq r$ with $t^{k'}_{i'} > t^{k'}_{i}$. 
We now define a configuration $(\bmu,\bs)\in \R$. We do so by describing $\eta(\mu^{(l)},s_l) = ({\boldsymbol \sigma}^l,{\boldsymbol r}^l)$ for $1 \leq l \leq r$. 
Set ${\boldsymbol \sigma}^l = \brho^l$ and ${\boldsymbol r}^l = \bs^l$ for $l \neq k,k'$. Set
\begin{align*}
r^k_j & = t^k_j \text{ for } j \neq i,i', & r^k_i & = t^k_i -1, & r^k_{i'} & = t^k_{i'}+1, & {\boldsymbol \sigma}^k & = (\emp,\emp,\ldots,\emp)\\
r^{k'}_j & = t^{k'}_j \text{ for } j \neq i,i', & r^{k'}_i & = t^{k'}_i + 1, & r^{k'}_{i'} &= t^{k'}_{i'}-1 & {\boldsymbol \sigma}^{k'} & = (\emp,\ldots,\emp,(w),\emp,\ldots,\emp)
\end{align*}
where $(w)$ occurs in position $i$ of $\sigma^{k'}$ and 
\[w=t^k_i-t^k_{i'}+t^{k'}_{i'}-t^{k'}_i-2 \geq t^{k'}_{i'}-t^{k'}_i >0.\] By construction $(\bla,\bs) \xrightarrow{e}_2 (\bmu,\bs)$ so $(\bmu,\bs)\in \R$ and hence is a Rouquier multipartition. 
Because $(\bmu,\bs)$ has an $e$-rim hook, Lemma~\ref{IfNotCore} implies that 
\[r_{i'}^{k'} \geq r_{i'-1}^{k'} \geq \ldots \geq r_{i+1}^{k'}\]
so
\[w \geq t^{k'}_{i'}-t^{k'}_i = r_{i'}^{k'}-r_i^{k'}+2 \geq r_{i+1}^{k'} - r_i^{k'} +2\]
contradicting the assumption that $(\bmu,\bs)$ is a Rouquier multipartition. 
\end{proof}

Suppose that $0 \leq i \leq e-1$. Define $\M(i) \in \Z^e_{\geq 0}$ by setting $M_j=0$ for $0 \leq j \leq i-1$ and $M_j=1$ for $i \leq j \leq e-1$.

\begin{lemma} \label{L22}
  Let $0 \leq i \leq e-1$. Suppose that $\R$ is a Rouquier block in $\A^r_e$ such that for every $(\bla,\bs) \in \R$ if $\eta(\la^{(k)},s_k) = (\brho^k,\bt^k)$
  for $1 \leq k \leq r$ then \[t^k_y + 1 - t^k_x \geq w^k\]
  for all $0 \leq x \leq i-1$ and all $i \leq y \leq e-1$, 
  where $w^k = |\brho^k|$. 
Set $\M=\M(i)$ and let $\R'=\Str(\R;\M)$. Define $\Pi:\R \rightarrow \R'$ by $\Pi(\bla,\bs)=\Str((\bla,\bs);\M)$ for $(\bla,\bs) \in \R$. Then $\R$ and $\R'$ are Morita equivalent and $\Pi$ is a decomposition equivalence map between $\R$ and $\R'$. 
\end{lemma}

\begin{proof}
The proof follows very similar lines to the proof of Lemma~\ref{baseR}. We have
\[\Pi = \Upsilon_0^{i} \circ \ldots \circ \Upsilon_{e-i-2}^{i}\circ\Upsilon_{e-i-1}^{i}\circ\upsilon^{e-i},\]
so to prove the result we need to show that each map $\Upsilon_j$ above is a decomposition equivalence map. Let $(\bla,\bs) \in \R$ and for each $1 \leq  k\leq r$, define $(\brho^k,\bt^k)$ and $w^k$ as in the statement of the lemma. 
Then each $\Upsilon_j$ is a decomposition equivalence map provided that $t_y^k+1 - t_x^k \geq w^k$ for all $1 \leq k \leq r$, $0 \leq x \leq i-1$ and $i \leq y \leq e-1$ which is the condition we assumed in the lemma. Morita equivalence follows from Theorem~\ref{C:Scopes} and Lemma~\ref{Applyupsilon2}. 
\end{proof}

\begin{proposition} \label{Padd}
Let $0 \leq i \leq e-1$. Suppose that $\R$ is a Rouquier block in $\A^r_e$ which is not a core block. Let $\M=\M(i)$ and $\R'=\Str(\R;\M)$. Define $\Pi:\R \rightarrow \R'$ by $\Pi(\bla,\bs)=\Str((\bla,\bs);\M)$ for $(\bla,\bs) \in \R$. Then $\R'$ is a Rouquier block which is Morita equivalent to $\R$ and $\Pi$ is a decomposition equivalence map between $\R$ and $\R'$. 
\end{proposition}

\begin{proof}
It follows from the definitions that $\R'$ is a Rouquier block. 
  By Lemma~\ref{L22}, it is then sufficient to show that if $(\bla,\bs) \in\R$ with
$\eta(\la^{(k)},s_k)=(\brho^k,\bt^k)$ 
for $1 \leq k \leq r$ and $w_k = |\brho^k|$ then
  $t_y^k+1 - t_x^k \geq w^k$ for all $1 \leq k \leq r$, $0 \leq x \leq i-1$ and $i \leq y \leq e-1$. If $w^k \geq 1$ then by Lemma~\ref{IfNotCore}
\[t_{e-1}^k+1 \geq \ldots \geq t^k_{i}+1 \geq t^k_{i-1} + w^k \geq \ldots \geq t^k_0 + w^k\]
as required. If $w^k=0$ then by Lemma~\ref{NotCoreM1} 
\[t^k_y+1 - t^k_x \geq 0 = w^k\]
and we are done. 
\end{proof}

We now look at core blocks. We first state a result of Fayers~\cite[Theorem~3.1]{Fayers:Cores}.

\begin{lemma}
  Suppose that $\R$ is a a core block. Then there exists ${\boldsymbol b} = (b_0,b_1,\ldots,b_{e-1}) \in \Z^{e-1}$ such that if $(\bla,\bs) \in \R$ and $\eta(\la^{(k)},s_k) = (\brho^k,\bt^k)$
for $1 \leq k \leq r$ then $t_i^k = b_i + \delta_{i,k}$ where $\delta_{i,k} \in \{0,1\}$, for $0 \leq i \leq e-1$ and $1 \leq k \leq r$. The tuple ${\boldsymbol b}$ is called a base tuple. 
\end{lemma}

We say that a block $\R\subset \A^r_e$ is a staircase block if $\R$ is a core block such that each $(\bla,\bs) \in \R$ satisfies the condition that if $\eta(\la^{(k)},s_k) = (\brho^k,\bt^k)$ for $1 \leq k \leq r$ then $t^k_0 \leq t^k_1 \leq \ldots \leq t_{e-1}^k$. 

\begin{lemma} \label{Stair2}
Suppose that $\R$ is a staircase block in $\A^e_r$. Let $0 \leq i \leq e-1$. Let $\M =\M(i)$ and $\R'=\Str(\R;\M)$. Define $\Pi:\R \rightarrow \R'$ by $\Pi(\bla,\bs)=\Str((\bla,\bs);\M)$ for $(\bla,\bs) \in \R$. Then $\R'$ is also a staircase block which is Morita equivalent to $\R$ and $\Pi$ is a decomposition equivalence map between $\R$ and $\R'$. 
\end{lemma}

\begin{proof}
That $\R'$ is a staircase block follows from the definition. We can apply Lemma~\ref{L22} to $\R$ since a staircase block satisfies the conditions of that lemma.  
\end{proof}

\begin{lemma} \label{CtoS}
Suppose that $\R$ is a core block in $\A^e_r$ which is not a staircase block. Then $\R$ is decomposition and Morita equivalent to a staircase block. 
  \end{lemma}

\begin{proof}
  Suppose that ${\boldsymbol b}=(b_0,b_1,\ldots,b_{e-1})$ is a base tuple for $\R$. Let $(\bla,\bs) \in \R$ and suppose that
$\eta(\la^{(k)},s_k) = (\brho^k,\bt^k)$
  for $1 \leq k \leq r$. Suppose that $b_{i}>b_{i+1}$ for some $0 \leq i < e-1$. Then $t^k_i \geq t^k_{i+1}$ for all $1 \leq k \leq r$, so by Proposition~\ref{DA}, $\Upsilon_i$ is a decomposition equivalence map. So we may assume that $b_0 \leq b_1 \leq \ldots \leq b_{e-1}$. Suppose that there exists $1 \leq k \leq r$ and $1 \leq i \leq e-1$ such that $t_{i-1}^k>t_{i}^k$.
  Then $b_{i-1}=b_{i}$ and so $t_{i}^k+1 \geq t_{i-1}^k$ for all $k$. Furthermore if $0 \leq x \leq i-1$ and $i \leq y \leq e-1$ then $b_x \leq b_y$ and so $t_y^k+1 \geq t_x^k$ for $1 \leq k \leq r$. By Lemma~\ref{L22}, we get that $\R$ is decomposition equivalent to $\Str(\R;\M(i))$. Repeating the argument as necessary, we obtain a staircase block $\R'$ which is decomposition equivalent to $\R$.  
\end{proof}

\begin{theorem} \label{TEquiv}
Suppose that $\R \subset \A^r_e$ is a core block or a Rouquier block. Then $\R$ is decomposition and Morita equivalent to a Rouquier block $\R'$ which has the property that $\Psi_r(\R')$ lies in an $r$-Rouquier block. 
\end{theorem}

\begin{proof}
  Suppose that $\R$ is a core block. By Lemma~\ref{CtoS}, $\R$ is decomposition equivalent to a staircase block. Appying Lemma~\ref{Stair2} multiple times, we see that any staircase block is decomposition equivalent to a block which satisfies the conditions of Lemma~\ref{GGR}.

  Now suppose that $\R$ is a Rouquier block which is not a core block. Then applying Proposition~\ref{Padd} multiple times, we see that $\R$ is decomposition equivalent to a block which satisfies the conditions of Lemma~\ref{GGR}.

That $\R'$ itself is Rouquier follows from Lemma~\ref{NotRR}.  
\end{proof}

We end by noting some equivalences between Rouquier blocks. 

\begin{theorem} \label{DMRR}
Suppose that $\R$ and $\R'$ are Rouquier blocks such that $\R'=\Str(\R;{\boldsymbol M})$ for some ${\boldsymbol M} \in \Z^e$. Define $\Pi:\R \rightarrow \R'$ by $\Pi(\bla,\bs)=\Str((\bla,\bs);\M)$ for $(\bla,\bs) \in \R$. Then $\R$ and $\R'$ are Morita equivalent and $\Pi$ is a decomposition equivalence map between $\R$ and $\R'$. 
\end{theorem}

We note that technically speaking the stretching operator may not give a Scopes equivalence as it may shift the multicharge, and as such we have avoided using that terminology above. If $r=1$ then for each $w\geq 0$ there is a unique Rouquier block (up to Scopes equivalence). We would like some way of indexing the Rouquier blocks when $r\geq 2$ up to equivalence.

\end{document}